\documentclass[a4paper]{article}
\usepackage{amssymb,amsmath,amsfonts,epsfig}
\usepackage{amsthm}

\catcode`@=11 \@addtoreset{equation}{section}
\renewcommand\theequation{\thesection.\@arabic\c@equation}
\catcode`@=12

\def\theequation{\thesection.\arabic{equation}}
\usepackage{amsmath,amsfonts,amssymb}
\usepackage{booktabs}
\usepackage{siunitx}        
\usepackage{multirow}       
\usepackage[table]{xcolor}  
\usepackage[title]{appendix}
\usepackage{bm}
\usepackage{algorithm}
\usepackage{amsmath,amsfonts,epsfig}
\usepackage{graphics}
\usepackage{supertabular}
\usepackage{amsthm}
\usepackage{mathtools}
\usepackage{amsfonts}
\usepackage{subfig}
\usepackage{tikz}
\usepackage{pgfplots}
\pgfplotsset{compat=1.18}

\usepackage{graphicx}
\usepackage{epstopdf}
\usepackage[numbers]{natbib}
\usepackage{float,epsfig, floatflt,here}
\usepackage{color}
\usepackage[colorlinks=true,citecolor=blue,linkcolor=blue]{hyperref}
\usepackage{fancyhdr}
\usepackage{etoolbox}
\usepackage{dsfont }
\usepackage{accents}
\newtheorem{remk}{\bf Remark}[section]

\usepackage{mathrsfs}
\usepackage{ upgreek }

\newcommand{\tnorm}[1]{{\left\vert\kern-0.25ex\left\vert\kern-0.25ex\left\vert #1\right\vert\kern-0.25ex\right\vert\kern-0.25ex\right\vert}}

\newcommand{\vertiii}[1]{{\left\vert\kern-0.25ex\left\vert\kern-0.25ex\left\vert #1
		\right\vert\kern-0.25ex\right\vert\kern-0.25ex\right\vert}}

\newtheorem{thm}{Theorem}[section]

\newtheorem{lemma}{Lemma}[section]

\renewcommand{\theequation}{A.\arabic{equation}}
\numberwithin{equation}{section}
\numberwithin{figure}{section}
\numberwithin{table}{section}
\title{Lyapunov Stability and Optimal Error Estimates for an SIPG Method for Weakly Damped Semilinear Wave Equations}
	\author{Ajeet Singh\footnote{Department of Mathematics, Indian Institute of Technology, Palaj, Gandhinagar, 382055, Gujarat, India \texttt{asingh@iitgn.ac.in}} \and  Abhinav Jha\footnote{Department of Mathematics, Indian Institute of Technology, Palaj, Gandhinagar, 382055, Gujarat, India \texttt{abhinav.jha@iitgn.ac.in}}}
\usepackage[a4paper, total={6.5in, 8.5in}]{geometry}
\usepackage{mathrsfs}
\begin{document}
\maketitle

\begin{abstract}
We develop and analyze a fully discrete scheme for the weakly damped
semilinear wave equation that combines a Symmetric Interior Penalty
Discontinuous Galerkin (SIPG) spatial discretization with a hybrid
Crank--Nicolson/second-order Backward Differentiation Formula
(CN--BDF2) time integrator.  A chord-slope linearization of the
nonlinear reaction term is employed, which preserves an exact discrete
gradient structure and, crucially, requires {no global Lipschitz
continuity assumption} on the nonlinearity.  Stability of the
fully discrete solution is established through a {Lyapunov-based}
analysis-rather than spectral arguments-by constructing a discrete
Lyapunov functional that yields existence, uniqueness, and
uniform boundedness of the numerical solution.  Under standard
regularity assumptions, optimal a~priori error estimates of order
$\mathcal{O}(h^{k}+\tau^{2})$ in the DG energy norm and
$\mathcal{O}(h^{k+1}+\tau^{2})$ in the $L^{2}$-norm are proved,
where $h$ is the mesh size, $\tau$ the time step, and $k$ the
polynomial degree.  Numerical experiments on two-dimensional problems
with linear, cubic, and trigonometric nonlinearities confirm the
theoretical convergence rates and illustrate the long-time
energy-dissipation properties guaranteed by the Lyapunov structure.
\end{abstract}

\section{Introduction}
\label{sec:intro}

Semilinear wave equations and nonlinear Klein--Gordon (KG) type
models arise in a broad range of scientific and engineering
applications, including nonlinear optics, plasma physics, acoustics,
fluid dynamics, seismic wave propagation, relativistic quantum
mechanics, and nonlinear elastic materials~\cite{arrieta1992damped,ball2004global}.
Of particular interest are \emph{weakly damped} semilinear wave
models, where a dissipative term $\sigma u_t$ gradually reduces the
wave energy while preserving the underlying hyperbolic oscillatory
structure. The strong coupling between the hyperbolic wave operator
and the nonlinear reaction makes the construction of stable, accurate,
and robust numerical schemes a challenging and active research topic.

\subsection{Scope and Problem Formulation}

We consider the weakly damped semilinear wave equation
\begin{equation}\label{eq:model}
  u_{tt} + \sigma u_t - \nabla\cdot(\kappa\nabla u) + g(u) = f
  \quad\text{in }\Omega\times(0,T],
\end{equation}
subject to homogeneous Dirichlet boundary and initial conditions
\begin{equation}\label{eq:bc-ic}
  u = 0 \;\text{ on }\;\partial\Omega\times(0,T],\qquad
  u(\cdot,0) = u_0,\quad u_t(\cdot,0) = u_1 \;\text{ in }\;\Omega.
\end{equation}
The following assumptions are imposed throughout this work.
\begin{itemize}
  \item[(A1)] The domain $\Omega\subset\mathbb{R}^d$, $d\ge 2$,
    is a bounded polygonal domain with the Lipschitz boundary $\partial \Omega$.
  \item[(A2)] The diffusion coefficient satisfies
    $0<\kappa_0\le\kappa(x)\le\kappa_1$ for all $x\in\Omega$.
  \item[(A3)] The damping parameter satisfies $\sigma>0$.
  \item[(A4)] The nonlinear term $g\in C^1(\mathbb{R})$ satisfies the growth and monotonicity conditions.  Its primitive $F$ (with $F'=g$)
    satisfies $F(s)\ge -c_1$ for all $s$, for some constant $c_1$.
  \item[(A5)] The forcing term and exact solution satisfy
    $f\in L^2(0,T;L^2(\Omega))$ and
\end{itemize}

To facilitate a second-order BDF2 time discretization, we recast
\eqref{eq:model} as a first-order system by introducing $v:=u_t$:
\begin{equation}\label{eq:intro-system}
  u_t = v,\qquad
  v_t + \sigma v - \nabla\cdot(\kappa\nabla u) + g(u) = f
  \quad\text{in }\Omega\times(0,T],
\end{equation}
with initial data $(u(\cdot,0),v(\cdot,0))=(u_0,u_1)$.
The weak formulation of \eqref{eq:intro-system} seeks
$(u(t),v(t))\in H_0^1(\Omega)\times L^2(\Omega)$ such that
\begin{equation}\label{eq:weak}
\begin{cases}
    (v_t,\phi)+\sigma(v,\phi)+(\kappa\nabla u,\nabla\phi)+(g(u),\phi)
  =(f,\phi),\quad\forall\phi\in H_0^1(\Omega),\\
  (u_t,\psi)=(v,\psi) ~~\forall \psi\in L^2(\Omega)
\end{cases}
\end{equation}
\subsection{Literature Overview}

Extensive numerical studies have been devoted to semilinear wave
and KG equations over the past several decades.
Finite difference time-domain schemes for KG equations were studied
in \cite{ablowitz1979solitary,han2008split}, while finite element approximations were
introduced by Kirby et al.~\cite{kirby2013galerkin}, who proved optimal convergence
in the energy norm.
Bao et al.~\cite{bao2012analysis} derived optimal energy estimates and
sub-optimal $L^2$-norm bounds in one dimension.
Energy-preserving standard and mixed finite element methods were
analyzed in \cite{he2020energy}, and non-conforming approximations for the
sine-Gordon equation were studied in \cite{shi2013nonconforming}.
Weak Galerkin methods combined with Newmark time stepping were
analyzed in \cite{jana2024weak}. High-order numerical methods on polygonal meshes for semilinear
reaction--diffusion models, including Sobolev equations and the
Allen--Cahn model, were analyzed via the hybrid high-order (HHO)
framework in \cite{singh2025high,kumar2025error}.
A VEM for the linear weakly damped problem was proposed in
\cite{pradhan2023virtual}, and the existence, uniqueness, and long-time behavior
of the semilinear problem were investigated in \cite{arrieta1992damped,ball2004global}.
One-dimensional weakly damped semilinear equations were treated in
\cite{rincon2013numerical}, and optimal convergence for two-dimensional cubic
nonlinearities was shown by Achouri~\cite{achouri2022efficient}.
VEM formulations for weakly damped sine-Gordon equations were
developed in \cite{adak2020virtual}, and weak Galerkin methods with
explicit time stepping for damped wave problems were recently analyzed
in \cite{mohapatra2026numerical}. Optimal HHO error estimates and simulations for the FitzHugh--Nagumo
system were established in \cite{singh2025rigorous, maurya2026new}. For nonlinear transport and reaction-dominated problems,
stabilized finite element methodologies together with adaptive
refinement strategies have also been investigated in
\cite{MR4598790,MR5036959} and
Lyapunov-based a priori error estimates for the FitzHugh--Nagumo model
via interior penalty DG methods were derived in \cite{singh2026priori}.

Despite this rich body of work, Interior Penalty Discontinuous Galerkin (IPDG) methods for \emph{weakly damped
semilinear wave equations} in two and three dimensions have not been
analyzed.  The IPDG framework offers local conservation, geometric
flexibility on unstructured meshes, $hp$-adaptivity, and strong
stability properties for hyperbolic problems.  Motivated by these
advantages, the present work develops and analyzes an SIPG method for
\eqref{eq:model}.

A central analytical difficulty is the nonlinear reaction term. The numerical treatment of nonlinear partial differential equations
under weak regularity assumptions has attracted considerable
attention in recent years. In particular, stabilized finite element
approaches and residual-based analyses for nonlinear
convection--diffusion--reaction equations have been developed in
\cite{MR4271582,MR4781562,MR4757334,MR5036959},
demonstrating the importance of robust discretizations and
rigorous error control for nonlinear problems.

Many existing analyses require \emph{global} Lipschitz
continuity of $g$, excluding physically relevant polynomial
nonlinearities such as the cubic KG term. In the present work,
we relax this assumption and allow $g$ to satisfy only local growth
and monotonicity conditions. A discrete \emph{Lyapunov functional}
is constructed via the primitive $F$ of $g$ to establish uniform
boundedness of the fully discrete solution.

\subsection{Contributions}

The main contributions of this article are summarized below.
\begin{itemize}
  \item[(i)]
    We formulate a fully discrete SIPG scheme for the weakly damped
    semilinear wave equation \eqref{eq:model} in two and three space
    dimensions.  To the best of our knowledge, this is the
    first IPDG analysis for this class of problems.
  \item[(ii)]
We establish stability of the fully discrete scheme by developing a
novel \emph{Lyapunov-based} energy analysis that avoids both spectral
techniques and linearization-based arguments. The analysis relies on a
discrete Lyapunov functional specifically designed for the chord-slope
treatment of the nonlinearity. Unlike many existing approaches, the
proposed framework does not require a \textbf{global Lipschitz continuity
assumption} on the nonlinear reaction term~$g$; instead, only local
growth and monotonicity conditions are assumed. The discrete Lyapunov
functional also serves as the principal tool for establishing existence,
uniqueness, and uniform boundedness of the numerical solution.
  \item[(iii)]
    The CN--BDF2 time integrator (Crank--Nicolson at $n=1$, BDF2 for
    $n\ge2$) yields second-order temporal accuracy and couples
    naturally with the discrete Lyapunov structure and
    optimal a priori error estimates are derived:
    $\mathcal{O}(h^{k}+\tau^{2})$ in the DG energy norm and
    $\mathcal{O}(h^{k+1}+\tau^{2})$ in the $L^{2}$-norm, under
    standard regularity assumptions on the exact solution.
  \item[(iv)]
    Numerical experiments on linear, polynomial, and trigonometric
    nonlinearities confirm the theoretical rates and demonstrate the
    long-time energy-dissipation properties guaranteed by the Lyapunov
    framework.
\end{itemize}
The remainder of this article is organized as follows.
Section~\ref{sec:ipdg} develops the SIPG spatial discretization.
Section~\ref{sec:scheme} presents the fully discrete CN--BDF2 IPDG
formulation and proves discrete energy stability.
Section~\ref{sec:error} derives optimal a priori error estimates.
Numerical results are in Section~\ref{sec:numerics}, and conclusions
in Section~\ref{sec:conclusion}.

\section{Interior Penalty Discontinuous Galerkin Discretization}
\label{sec:ipdg}
\subsection{Discrete setting and notation}

Throughout the article, $\mathcal{T}_h$ stands for a quasi-uniform, shape-regular
triangulation of the domain $\Omega$ into triangles (2D) or tetrahedra (3D),
so that
$
\overline{\Omega} = \bigcup_{K\in\mathcal{T}_h}\overline{K}.
$
For each element $K\in\mathcal{T}_h$, the symbol $h_K$ refers to its
diameter, and the global mesh parameter is
$
h := \max_{K\in\mathcal{T}_h} h_K.
$

The broken Sobolev space and the discontinuous Galerkin polynomial space
associated with $\mathcal{T}_h$ are defined by
\[
  H^s(\mathcal{T}_h)
  := \bigl\{\,w\in L^2(\Omega)\;:\;
     w|_K\in H^s(K)\ \text{for every}\ K\in\mathcal{T}_h\bigr\},
\]
\[
  V_h
  := \bigl\{\,w_h\in L^2(\Omega)\;:\;
     w_h|_K\in\mathbb{P}_k(K)\ \text{for every}\ K\in\mathcal{T}_h\bigr\},
\]
where $\mathbb{P}_k(K)$ collects all polynomials of total degree at most
$k$ restricted to $K$.

We employ $\mathcal{E}_h^I$ and $\mathcal{E}_h^B$ for the collections of
interior and boundary faces of $\mathcal{T}_h$, respectively,
and set $\mathcal{E}_h:=\mathcal{E}_h^I\cup\mathcal{E}_h^B$.
The element-wise $L^2$-inner product is
\[
  (w,z)_{\mathcal{T}_h}
  := \sum_{K\in\mathcal{T}_h}\int_K w\,z\;dx,
\]
and, for any subfamily $\mathcal{S}_h\subset\mathcal{E}_h$, the
face-wise inner product reads
\[
  \langle w,z\rangle_{\mathcal{S}_h}
  := \sum_{e\in\mathcal{S}_h}\int_e w\,z\;ds.
\]

For a pair of elements $K^+,K^-\in\mathcal{T}_h$ sharing an interior face
$e=\partial K^+\cap\partial K^-$, the face unit normal is fixed as
$\mathbf{n}_e:=\mathbf{n}_{K^+}\big|_e=-\mathbf{n}_{K^-}\big|_e$.
The standard DG jump and average operators are then given, for
$w\in H^1(\mathcal{T}_h)$, by
\[
  [w] := w\big|_{K^+}-w\big|_{K^-}
  \ \text{on}\ e\in\mathcal{E}_h^I,
  \qquad
  [w] := w
  \ \text{on}\ e\in\mathcal{E}_h^B,
\]
\[
  \{w\} := \tfrac12\bigl(w\big|_{K^+}+w\big|_{K^-}\bigr)
  \ \text{on}\ e\in\mathcal{E}_h^I,
  \qquad
  \{w\} := w
  \ \text{on}\ e\in\mathcal{E}_h^B.
\]

\subsection{SIPG bilinear form and DG norm \cite{singh2026priori}}

The symmetric interior penalty Galerkin (SIPG) bilinear form is
\begin{align}\label{eq:ah}
  {\cal A}_h(w_h,z_h)
  &= \sum_{K\in\mathcal{T}_h}\int_K
      \kappa\,\nabla w_h\cdot\nabla z_h\;dx
    -\sum_{e\in\mathcal{E}_h}\int_e
      \{\kappa\nabla w_h\}\cdot[z_h]\;ds
    -\sum_{e\in\mathcal{E}_h}\int_e
      \{\kappa\nabla z_h\}\cdot[w_h]\;ds \notag\\
  &\quad+\sum_{e\in\mathcal{E}_h}
      \frac{\eta}{h_e}\int_e [w_h]\cdot[z_h]\;ds,
\end{align}
and the associated DG norm is
\begin{equation}\label{eq:dg-norm}
  \|w_h\|_{DG}^2
  := \sum_{K\in\mathcal{T}_h}\|\nabla w_h\|_K^2
    +\sum_{e\in\mathcal{E}_h}\frac{\eta}{h_e}\,\|[w_h]\|_e^2.
\end{equation}

\begin{lemma}[Coercivity]\label{lem:coercivity}
For a sufficiently large penalty parameter $\eta>0$,
there exists a constant $\alpha_0>0$, independent of $h$,
such that
\[
  {\cal A}_h(w_h,w_h) \;\ge\; \alpha_0\,\|w_h\|_{DG}^2,
  \qquad \forall\,w_h\in V_h.
\]
\end{lemma}

\subsection{$L^2$-projection and approximation estimates \cite{di2011mathematical}}

\paragraph{Element-wise estimate on a shape-regular mesh.}
For each $K\in\mathcal{T}_h$, let
$\Pi_h:L^2(K)\to\mathbb{P}_k(K)$
denote the $L^2$-orthogonal projector onto polynomials of degree at most
$k$. If $w|_K\in H^{s_K+1}(K)$ with $s_K\ge0$ and
$t_K:=\min\{s_K,k\}$, then for $m\in\{0,1\}$,
\begin{equation}\label{eq:local-approx}
  \|w-\Pi_h w\|_{H^m(K)}
  \;\le\; C\,h_K^{\,t_K+1-m}\,|w|_{H^{t_{K}+1}(K)},
  \qquad \forall\,K\in\mathcal{T}_h.
\end{equation}
When $w|_K\in H^{k+1}(K)$ (i.e.\ $s_K\ge k$), the above reduces to
\begin{equation}\label{eq:local-optimal}
  \|w-\Pi_h w\|_{H^m(K)}
  \;\le\; C\,h_K^{\,k+1-m}\,|w|_{H^{k+1}(K)},
  \qquad m=0,1.
\end{equation}

\paragraph{Global rates under quasi-uniformity.}
Assuming additionally that $\mathcal{T}_h$ is quasi-uniform with mesh size
$h$ and $w\in H^{k+1}(\Omega)$, a summation of \eqref{eq:local-optimal}
over $\mathcal{T}_h$ produces the global bounds
\begin{equation}\label{eq:global-estimates}
  \|w-\Pi_h w\|_{L^2(\Omega)}
  \;\le\; C\,h^{k+1}\,|w|_{H^{k+1}(\Omega)},
  \qquad
  |w-\Pi_h w|_{H^1(\Omega)}
  \;\le\; C\,h^k\,|w|_{H^{k+1}(\Omega)}.
\end{equation}
Hence the $L^2$-projection attains the optimal rate
$\mathcal{O}(h^{k+1})$ in the $L^2$-norm and
$\mathcal{O}(h^k)$ in the $H^1$-seminorm.

\noindent\textbf{Inverse Inequality \cite{di2011mathematical}.}
For any $w_h\in V_h$, there exists a constant $C_{\mathrm{inv}}>0$,
depending only on the polynomial degree $k$ and the shape-regularity
of the mesh, such that
\begin{equation}\label{eq:inverse-ineq}
  \|w_h\|_{DG}\le C_{\mathrm{inv}}\,h^{-1}\|w_h\|,
  \qquad
  {\cal A}_h(w_h,\phi_h)\le C_{\mathrm{inv}}^2\,h^{-2}\|w_h\|\,\|\phi_h\|
  \quad\forall\,w_h,\phi_h\in V_h.
\end{equation}
The second bound follows from the first together with
the continuity of ${\cal A}_h$ on $V_h\times V_h$:
\[
  |{\cal A}_h(w_h,\phi_h)|
  \le \|w_h\|_{DG}\,\|\phi_h\|_{DG}
  \le C_{\mathrm{inv}}^2\,h^{-2}\|w_h\|\,\|\phi_h\|.
\]
\section{Fully Discrete CN--BDF2 SIPG Scheme}
\label{sec:scheme}

\noindent\textbf{SIPG semi-discrete form.}
Find $(u_h,v_h)\in V_h\times V_h$ such that for all $\phi_h,\psi_h\in V_h$:
\begin{equation}\label{eq:semi}
  (\partial_t u_h,\psi_h)=(v_h,\psi_h),\qquad
  (\partial_t v_h,\phi_h)+\sigma(v_h,\phi_h)+{\cal A}_h(u_h,\phi_h)+(g(u_h),\phi_h)=(f,\phi_h).
\end{equation}

\medskip\noindent\textbf{Temporal discretization and notation.}
For a positive integer $N$, let $\tau=T/N$ denote the time-step size
for the uniform partition
$0=t_0<t_1<\cdots<t_N=T$ of the interval $[0,T]$.
For a continuous function $\omega:[0,T]\to L^2(\Omega)$, we write
$\omega^n=\omega(\cdot,t_n)$ and define the backward difference
quotients
\[
  {\partial}_{\tau}\omega^n
  :=\frac{\omega^{n}-\omega^{n-1}}{\tau},\qquad
  \partial^2_{\tau}\omega^n
  :=\frac{\partial_{\tau}\omega^n-\partial_{\tau}\omega^{n-1}}{\tau},
  \qquad 2\le n\le N.
\]
The second-order backward differentiation formula (BDF2) operator is
then given by \cite{di2011mathematical}
\begin{equation}\label{eq:bdf2op}
  D^{(2)}_t{\omega}^n
  := \partial_{\tau}{\omega}^n+\frac{\tau}{2}\partial^2_{\tau}{\omega}^n
  = \frac{3\omega^n-4\omega^{n-1}+\omega^{n-2}}{2\tau},
  \qquad n\ge2,
\end{equation}
where the sequence $\{\omega^n\}_{n=0}^N \subset L^2(\Omega)$.
The midpoint average is denoted by
$\omega^{n-1/2}:=(\omega^n+\omega^{n-1})/2$.
The chord-slope nonlinearity ($F'=g$) \cite{acharya2025conservative}:
\begin{equation}\label{eq:chord}
  G(a,b):=\begin{cases}(F(a)-F(b))/(a-b),&a\ne b,\\ g(a),&a=b,\end{cases}
  \quad\text{so}\quad
  (G(a,b),\,a-b)=(F(a)-F(b),1). 
\end{equation}
\subsection{Preliminary results}
\begin{lemma}[Properties of the chord-slope operator $G$ \cite{acharya2025conservative}]\label{lem:G-properties}
Let $G$ be defined by \eqref{eq:chord} with $F'=g$.
\begin{itemize}
\item[\textup{(i)}] \textbf{Discrete gradient identity:}
For all $a,b\in V_h$,
\begin{equation}\label{eq:G-identity}
  (G(a,b),\,a-b) = (F(a)-F(b),\,1).
\end{equation}
\item[\textup{(ii)}] \textbf{MVT-bound:}
By the Mean Value Theorem, there exists a constant $L_g$ such that for $a_i,b_i$ ($i=1,2$) in a bounded subset of $L^\infty(\Omega)$,
\begin{equation}\label{eq:G-lip}
  \|G(a_1,b_1)-G(a_2,b_2)\|
  \le C_g\bigl(\|a_1-a_2\|+\|b_1-b_2\|\bigr),
\end{equation}
where $C_g>0$ is a constant. 
\item[\textup{(iii)}] \textbf{Consistency with $g$:}
For a smooth function $w$ and $\hat\partial_t w^n:=(w^{n+1}-w^{n-1})/(2\tau)$,
\begin{equation}\label{eq:G-consistency}
  \|G(w^{n+1},w^{n-1})-g(w^n)\|
  \le C\tau^2\|w_{ttt}\|_{L^\infty(t_{n-1},t_{n+1};L^\infty)}.
\end{equation}
\end{itemize}
\end{lemma}

\begin{proof}
Part~(i) follows directly from the definition \eqref{eq:chord}.

For part~(ii), since $G(x,y)=(F(x)-F(y))/(x-y)$ for $x\ne y$,
the mean value theorem gives $G(x,y)=g(\zeta)$ for some $\zeta$
between $x$ and $y$. A detailed algebraic expansion (cf.~\cite[Lemma~4.4]{acharya2025conservative})
yields
\[
  G(x_1,y_1)-G(x_2,y_2)
  = (x_1-x_2)\,P(x_1,x_2,y_1,y_2)
  + (y_1-y_2)\,Q(x_1,x_2,y_1,y_2),
\]
where $P,Q$ are polynomials bounded on bounded sets,
and the $L^2$-norm estimate \eqref{eq:G-lip} follows by H\"older's inequality.

For part~(iii), Taylor expansion about $t_n$ gives
$w^{n\pm1}=w^n\pm\tau w_t^n+\frac{\tau^2}{2}w_{tt}^n\pm\frac{\tau^3}{6}w_{ttt}^n+O(\tau^4)$.
Since $G(w^{n+1},w^{n-1})=g(w^n)+O(\tau^2)$ by the symmetric structure of the
chord slope, the bound \eqref{eq:G-consistency} follows.
\end{proof}

\begin{lemma}[Taylor expansion with integral remainder]\label{lem:taylor}
For $w\in H^3(t_n-\tau,t_n+\tau)$, there holds
\begin{equation}\label{eq:taylor-hat}
  \left\|\frac{w(t_n+\tau)-w(t_n-\tau)}{2\tau}-w'(t_n)\right\|^2
  \le C\tau^3\int_{t_n-\tau}^{t_n+\tau}\|w'''(s)\|^2\,\mathrm{d}s.
\end{equation}
For $w\in H^4(t_{n-1},t_{n+1})$, the BDF2 consistency error satisfies
\begin{equation}\label{eq:taylor-bdf2}
  \left\|D^{(2)}_t w^n - w'(t_n)\right\|^2
  \le C\tau^3\int_{t_{n-2}}^{t_n}\|w'''(s)\|^2\,\mathrm{d}s.
\end{equation}
\end{lemma}

\begin{proof}
By Taylor's expansion with integral remainder:
\[
  w(t_n\pm\tau)=w(t_n)\pm\tau w'(t_n)+\frac{\tau^2}{2}w''(t_n)
  +\int_{t_n}^{t_n\pm\tau}\frac{(t_n\pm\tau-s)^2}{2}w'''(s)\,\mathrm{d}s.
\]
Subtracting and dividing by $2\tau$, then squaring and applying
$(a+b)^2\le 2(a^2+b^2)$ with H\"older's inequality yields \eqref{eq:taylor-hat}.



Recall the BDF2 operator
$D^{(2)}_t w^n = (3w^n - 4w^{n-1} + w^{n-2})/(2\tau)$.
Applying Taylor's theorem with integral remainder about $t_n$ we can obtain \eqref{eq:taylor-bdf2} bound \cite{singh2025rigorous,singh2026priori}.
\end{proof}

\medskip\noindent
\textbf{Fully discrete CN--BDF2 SIPG scheme.}
Given $(u_h^0,v_h^0)=(\Pi_h u_0,\Pi_h v_0)$, find
$(u_h^n,v_h^n)\in V_h\times V_h$ for $n\ge1$:

\smallskip
\textit{$n=1$ (Crank--Nicolson initialization):}
Find $(u_h^1,v_h^1)$ such that for all $\phi_h,\psi_h\in V_h$,
\begin{subequations}\label{eq:cn}
\begin{align}
  \bigl(\partial_\tau u_h^1,\,\phi_h\bigr) &= \bigl(v_h^{1/2},\,\phi_h\bigr),
  \label{eq:cn-u}\\
  \bigl(\partial_\tau v_h^1,\,\psi_h\bigr)
  +\sigma\bigl(v_h^{1/2},\,\psi_h\bigr)
  +{\cal A}_h\bigl(u_h^{1/2},\,\psi_h\bigr)
  +\bigl(G(u_h^1,u_h^{-1}),\,\psi_h\bigr)
  &= \bigl(f^{1/2},\,\psi_h\bigr),
  \label{eq:cn-v}
\end{align}
\end{subequations}
where the midpoint average notation $\varphi^{1/2}:=(\varphi^1+\varphi^0)/2$
is used and $u_h^{-1}:=u_h^0-\tau v_h^0$.

\smallskip
\textit{$n\ge2$ (BDF2):}
Find $(u_h^n,v_h^n)$ such that for all $\phi_h,\psi_h\in V_h$,
\begin{subequations}\label{eq:bdf2}
\begin{align}
  \bigl(D^{(2)}_t u_h^n,\,\phi_h\bigr) &= \bigl(v_h^n,\,\phi_h\bigr),
  \label{eq:bdf2-u}\\
  \bigl(D^{(2)}_t v_h^n,\,\psi_h\bigr)
  +\sigma\bigl(v_h^n,\,\psi_h\bigr)
  +{\cal A}_h\bigl(u_h^n,\,\psi_h\bigr)
  +\bigl(G(u_h^n,u_h^{n-2}),\,\psi_h\bigr)
  &= \bigl(f^n,\,\psi_h\bigr).
  \label{eq:bdf2-v}
\end{align}
\end{subequations}
\medskip
Next, we introduce the Lyapunov functional defined by
\begin{equation}\label{eq:lyap}
  \mathcal{Z}_h^n=\mathcal{Z}_h(u_h^n) := \tfrac{1}{2}{\cal A}_h(u_h^n,u_h^n)+\tau(F(u_h^n),1),
\end{equation}
\begin{remk}[Role of the value $u_h^{-1}$]\label{rem:ghost}
The auxiliary quantity $u_h^{-1}:=u_h^0-\tau v_h^0$ is a first-order
backward extrapolation of the initial data to the fictitious time level
$t_{-1}=-\tau$; it is not an additional unknown but is fully determined
by $(u_h^0,v_h^0)$. Its purpose is to pair the Crank--Nicolson
nonlinearity $G(u_h^1,u_h^{-1})$ with the stride-two difference
$(u_h^1-u_h^{-1})/(2\tau)$, so that the discrete gradient identity
$(G(u_h^1,u_h^{-1}),\,u_h^1-u_h^{-1})=(F(u_h^1)-F(u_h^{-1}),1)$
holds exactly at $n=1$ and telescopes seamlessly into the BDF2 energy
at $n\ge2$. Since $u_h^{-1}$ approximates $u(-\tau)$ to
$\mathcal{O}(\tau^2)$, its introduction is consistent with the overall
second-order temporal accuracy of the CN-BDF2 scheme. Moreover, in
the uniqueness analysis, two solutions
sharing the same initial data necessarily share the same $u_h^{-1}$ value,
yielding $Z_u^{-1}=0$ and thereby eliminating the nonlinear residual
at $n=1$.
\end{remk}
\begin{lemma}[Stability]\label{lem:energy}
Let Assumptions \textup{(A1)--(A5)} hold, let $0<\sigma<2$, and let $0<\tau<2/3$.
Then the solution $(u_h^n,v_h^n)_{n\ge0}$ of \eqref{eq:cn}--\eqref{eq:bdf2}
satisfies, for all $N\ge1$,
\begin{equation}\label{eq:energy-id}
  \frac{\tau}{4}\|\partial_\tau v_h^N\|^2
  + \frac{\sigma}{2}\|v_h^N\|^2
  + \mathcal{Z}_h(u_h^N)
  \le \frac{2(2-\tau)}{2-3\tau}
  \Bigl(\frac{\tau}{2\varepsilon}\sum_{n=0}^{N}\|f^n\|^2
  + \frac{5+\tau\sigma}{4}\|v_h^0\|^2 + 2\mathcal{Z}_h^0\Bigr),
\end{equation}
$\mathcal{Z}_h^0:=\mathcal{Z}_h(u_h^0)$,
and the constant $\frac{2(2-\tau)}{2-3\tau}$ is independent of $h$
and uniformly bounded for $\tau\in(0,\tau_0)$ with any fixed $\tau_0<2/3$.
Consequently,
\begin{equation}\label{eq:apriori}
  \|v_h^N\|\le C,\quad \|u_h^N\|_{H^1(\Omega)}\le C,\quad
  \|u_h^N\|_{L^p(\Omega)}\le C,\quad 1\le p<\infty,
\end{equation}
where $C=C\bigl(\sigma,\|u^0\|,\|v^0\|,\|f\|\bigr)$
is independent of $h$ and $\tau$.
\end{lemma}

\begin{proof}

\noindent
Choose
$
\psi_h=v_h^{1/2}
$
in \eqref{eq:cn-v} and we obtain
\begin{equation}
    (\partial_\tau v_h^1,v_h^{1/2})+\sigma(v_h^{1/2},v_h^{1/2})+{\cal A}_h(u_h^{1/2},v_h^{1/2})+(G(u_h^1,u_h^{-1}),v_h^{1/2})
=(f^{1/2},v_h^{1/2}).
\end{equation}
First, we handle the first term of the left-hand side
\begin{align}
(\partial_\tau v_h^1,v_h^{1/2})
=
\frac1{\tau}
\left(
v_h^1-v_h^0,
\frac{v_h^1+v_h^0}{2}
\right)
=
\frac1{2\tau}
\left(
\|v_h^1\|^2-\|v_h^0\|^2
\right).
\label{eq:kinetic_cn}
\end{align}

Using the symmetry of the bilinear form
${\cal A}_h(\cdot,\cdot)$
and the identity
$v_h^{1/2}=\partial_\tau u_h^1$,
we obtain
\begin{align}
{\cal A}_h(u_h^{1/2},v_h^{1/2})
=
{\cal A}_h
\left(
\frac{u_h^1+u_h^0}{2},
\partial_\tau u_h^1
\right)
=
\frac1{2\tau}
\left(
{\cal A}_h(u_h^1,u_h^1)
-
{\cal A}_h(u_h^0,u_h^0)
\right).
\label{eq:elastic_cn}
\end{align}

For the nonlinear term,
using \eqref{eq:chord},
we have
\begin{align}
(G(u_h^1,u_h^{-1}),v_h^{1/2})
=
\frac1{2\tau}
(G(u_h^1,u_h^{-1}),
u_h^1-u_h^{-1})
+
\frac12
(G(u_h^1,u_h^{-1}),v_h^0).
\label{eq:nonlinear_split}
\end{align}

By the discrete gradient identity,
\begin{align}
(G(u_h^1,u_h^{-1}),
u_h^1-u_h^{-1})
=
(F(u_h^1)-F(u_h^{-1}),1).
\label{eq:disc_grad}
\end{align}

Substituting
\eqref{eq:kinetic_cn}--\eqref{eq:disc_grad}
into \eqref{eq:cn-v} and using the coercivity of ${\cal A}_h(\cdot,\cdot)$,
we arrive at
\begin{equation}
    \frac1{2\tau}
\left(
\|v_h^1\|^2-\|v_h^0\|^2
\right)+\sigma \|v_h^{1/2}\|^2+\frac1{2\tau}
\left(
\|u_h^1\|^2_{DG}
-
\|u_h^0\|^2_{DG}
\right)+(F(u_h^1)-F(u_h^0),1) = (f^{1/2},v^{1/2}),
\end{equation}
where we have used $u_h^{-1}=u_h^0-\tau v_h^0$ together with the smoothness of $F$.

Multiplying by $\tau$ and rearranging the terms, we obtain
\begin{align}\label{eq:cn1}
    \|v^1_h\|^2-\|v^0_h\|^2+{\cal Z}_h^1-{\cal Z}_h^0+2\tau\sigma\|v^{1/2}_h\|^2 &= 2\tau (f^{1/2},v^{1/2})\nonumber\\
   \|v^1_h\|^2-\|v^0_h\|^2+{\cal Z}_h^1-{\cal Z}_h^0+2\tau\sigma\|v^{1/2}_h\|^2 &\le \frac{1}{2}\|f^{1/2}\|^2+\frac{1}{4}\|v^1_h\|^2+\frac{1}{4}\|v_h^0\|^2 \nonumber\\
    \frac{3-\tau\sigma}{4}\|v^1_h\|^2+{\cal Z}_h^1&\le \frac{5+\tau\sigma}{4}\|v^0_h\|^2+{\cal Z}_h^0+\frac{1}{2}\|f^{1/2}\|^2\nonumber\\
    \frac{1}{4}\|v^1_h\|^2+{\cal Z}_h^1&\le \frac{5+\tau\sigma}{4}\|v^0_h\|^2+{\cal Z}_h^0+\frac{1}{2}\|f^{1/2}\|^2
\end{align}
This completes the estimate for the Crank--Nicolson starting step.

\bigskip
\noindent For $n\ge 2$, we employ BDF2.

Choose $\psi_h = \partial_\tau v_h^n$ in \eqref{eq:bdf2-v} to obtain
\begin{equation}\label{eq:bdf2-tested}
  \bigl(D^{(2)}_t v_h^n,\partial_\tau v_h^n\bigr)
  +\sigma\bigl(v_h^n,\partial_\tau v_h^n\bigr)
  +{\cal A}_h\bigl(u_h^n,\partial_\tau v_h^n\bigr)
  +\bigl(G(u_h^n,u_h^{n-2}),\partial_\tau v_h^n\bigr)
  = \bigl(f^n,\partial_\tau v_h^n\bigr).
\end{equation}
We treat each term on the left-hand side separately.
\smallskip
Applying the algebraic identity $(p-q,p) = \frac{1}{2}(\|p\|^2 - \|q\|^2) + \frac{1}{2}\|p-q\|^2
\ge \frac{1}{2}(\|p\|^2-\|q\|^2)$ with $p=\partial_\tau v_h^n$ and $q=\partial_\tau v_h^{n-1}$ yields
\begin{equation}\label{eq:bdf2-kin}
  \bigl(D^{(2)}_t v_h^n,\,\partial_\tau v_h^n\bigr)
  \ge \|\partial_\tau v_h^n\|^2
  + \frac{\tau}{4}\bigl(\|\partial_\tau v_h^n\|^2 - \|\partial_\tau v_h^{n-1}\|^2\bigr)
  = \|\partial_\tau v_h^n\|^2 + \frac{\tau}{4}\,\partial_\tau\|\partial_\tau v_h^n\|^2.
\end{equation}

\smallskip
Using the identity $a(a-b) = \frac{1}{2}(a^2-b^2) + \frac{1}{2}(a-b)^2$ with $a=v_h^n$
and $b=v_h^{n-1}$:
\begin{equation}\label{eq:bdf2-damp}
  \sigma\bigl(v_h^n,\partial_\tau v_h^n\bigr)
  = \frac{\sigma}{2\tau}\bigl(\|v_h^n\|^2 - \|v_h^{n-1}\|^2\bigr)
  + \frac{\sigma\tau}{2}\|\partial_\tau v_h^n\|^2
  \ge \frac{\sigma}{2\tau}\bigl(\|v_h^n\|^2 - \|v_h^{n-1}\|^2\bigr).
\end{equation}

\smallskip
From equation \eqref{eq:bdf2-u} we have $v_h^n = D^{(2)}_t u_h^n$, and consequently
$\partial_\tau v_h^n = \partial_\tau D^{(2)}_t u_h^n$.
Applying the same identity $a(a-b) = \frac{1}{2}(a^2-b^2)+\frac{1}{2}(a-b)^2$
and the symmetry of ${\cal A}_h$:
\begin{equation}\label{eq:bdf2-elas}
  {\cal A}_h\bigl(u_h^n,\partial_\tau v_h^n\bigr)
  \ge \frac{1}{2\tau}\bigl({\cal A}_h(u_h^n,u_h^n) - {\cal A}_h(u_h^{n-1},u_h^{n-1})\bigr).
\end{equation}
More precisely, writing $u_h^n - u_h^{n-2} = 2\tau\,\hat\partial_t u_h^{n-1}$ and using the
defining property \eqref{eq:chord}:
\begin{align}
  \bigl(G(u_h^n,u_h^{n-2}),\partial_\tau v_h^n\bigr)
  &= \frac{1}{2\tau}\bigl(G(u_h^n,u_h^{n-2}),\,u_h^n-u_h^{n-2}\bigr)
     + \bigl(G(u_h^n,u_h^{n-2}),\,\partial_\tau v_h^n - \tfrac{u_h^n-u_h^{n-2}}{2\tau}\bigr)
  \nonumber\\
  &= \frac{1}{2\tau}\bigl(F(u_h^n)-F(u_h^{n-2}),1\bigr) + R_h^n,
  \label{eq:nonlin-bdf2}
\end{align}
where
$R_h^n:=\bigl(G(u_h^n,u_h^{n-2}),\,
\partial_\tau v_h^n - \frac{u_h^n-u_h^{n-2}}{2\tau}\bigr)$.
Since $u_h^n-u_h^{n-1}=\tau v_h^n$ (from \eqref{eq:bdf2-u}),
we have
$\frac{u_h^n-u_h^{n-2}}{2\tau}=\frac{v_h^n+v_h^{n-1}}{2}
=v_h^n-\frac{\tau}{2}\partial_\tau v_h^n$,
hence
\begin{equation}\label{eq:Rhn-identity}
  R_h^n = \frac{\tau}{2}\bigl(G(u_h^n,u_h^{n-2}),\,\partial_\tau v_h^n\bigr).
\end{equation}

\smallskip
Substituting \eqref{eq:bdf2-kin}--\eqref{eq:nonlin-bdf2} into \eqref{eq:bdf2-tested}
and multiplying through by $\tau$:
\begin{align}
  &\tau\|\partial_\tau v_h^n\|^2
  + \frac{\tau^2}{4}\,\partial_\tau\|\partial_\tau v_h^n\|^2
  + \frac{\sigma}{2}\bigl(\|v_h^n\|^2-\|v_h^{n-1}\|^2\bigr)
  + \frac{1}{2}\bigl({\cal A}_h(u_h^n,u_h^n)-{\cal A}_h(u_h^{n-1},u_h^{n-1})\bigr) \nonumber\\
  &\quad + \frac{1}{2}\bigl(F(u_h^n)-F(u_h^{n-2}),1\bigr)
  \le \tau\,R_h^n + \tau(f^n,\partial_\tau v_h^n).
  \label{eq:collect-raw}
\end{align}

Adding and subtracting $\bigl(F(u_h^{n-1}),1\bigr)$ on the left-hand side of \eqref{eq:collect-raw}
and recalling the definition \eqref{eq:lyap} of $\mathcal{Z}_h$:
\begin{align}
  &\tau\|\partial_\tau v_h^n\|^2
  + \frac{\tau^2}{4}\,\partial_\tau\|\partial_\tau v_h^n\|^2
  + \frac{\sigma}{2}\bigl(\|v_h^n\|^2-\|v_h^{n-1}\|^2\bigr)+ \bigl[\mathcal{Z}_h(u_h^n)-\mathcal{Z}_h(u_h^{n-1})\bigr]\nonumber\\
  &\le \tau\,R_h^n + \tau(f^n,\partial_\tau v_h^n) +  \frac{1}{2}\bigl(F(u_h^{n-1}) - F(u_h^{n-2}),1\bigr).
  \label{eq:collect-Z}
\end{align}
Next, summing \eqref{eq:collect-Z} over $n$ from $2$ to $N$,
we obtain
\begin{align}
  &\tau\sum_{n=2}^{N}\|\partial_\tau v_h^n\|^2
  + \frac{\tau^2}{4}\sum_{n=2}^{N}\partial_\tau\|\partial_\tau v_h^n\|^2
  + \frac{\sigma}{2}\sum_{n=2}^{N}
  \bigl(\|v_h^n\|^2-\|v_h^{n-1}\|^2\bigr)
  + \sum_{n=2}^{N}
  \bigl[\mathcal{Z}_h(u_h^n)-\mathcal{Z}_h(u_h^{n-1})\bigr]
\nonumber\\
  &~~~~~~~\le
  \tau\sum_{n=2}^{N} R_h^n
  +
  \tau\sum_{n=2}^{N}
  (f^n,\partial_\tau v_h^n)
  +
  \frac12
  \sum_{n=2}^{N}
  \bigl(
  (F(u_h^{n-1})-F(u_h^{n-2}),1)
  \bigr).
\label{eq:summed-Z}
\end{align}
Using the telescoping property, it follows that
\begin{align}
  &\tau\sum_{n=2}^{N}\|\partial_\tau v_h^n\|^2
  + \frac{\tau}{4}
  \Bigl(
  \|\partial_\tau v_h^N\|^2
  -
  \|\partial_\tau v_h^1\|^2
  \Bigr)
  +
  \frac{\sigma}{2}
  \Bigl(
  \|v_h^N\|^2
  -
  \|v_h^1\|^2
  \Bigr)
  +
  \mathcal{Z}_h(u_h^N)
  -
  \mathcal{Z}_h(u_h^1)
\nonumber\\
  &~~~~~~~~~~\le
  \tau\sum_{n=2}^{N} R_h^n
  +
  \tau\sum_{n=2}^{N}
  (f^n,\partial_\tau v_h^n)
  +
  \frac12
  \Bigl(
  (F(u_h^{N-1}),1)
  -
  (F(u_h^{0}),1)
  \Bigr).
\label{eq:telescopic-Z}
\end{align}

\noindent Next, we estimate the first term of the right-hand side $\tau\sum R_h^n$ 
by combining \eqref{eq:nonlin-bdf2} and \eqref{eq:Rhn-identity},
we obtain
\[
  \Bigl(1-\frac{\tau}{2}\Bigr)(G,\partial_\tau v_h^n)
  = \frac{1}{2\tau}(F(u_h^n)-F(u_h^{n-2}),1),
\]
hence
\begin{equation}\label{eq:Rhn-closed}
  R_h^n = \frac{1}{2(2-\tau)}\bigl(F(u_h^n)-F(u_h^{n-2}),1\bigr).
\end{equation}
Multiplying by $\tau$ and summing over $n=2,\ldots,N$,
the stride-2 telescoping gives
\begin{align}\label{eq:R-telescoped}
  \tau\sum_{n=2}^{N}R_h^n
  &= \frac{\tau}{2(2-\tau)}\sum_{n=2}^{N}\bigl(F(u_h^n)-F(u_h^{n-2}),1\bigr)\nonumber\\
  &= \frac{\tau}{2(2-\tau)}\Bigl[(F(u_h^N),1)+(F(u_h^{N-1}),1)
     -(F(u_h^0),1)-(F(u_h^1),1)\Bigr].
\end{align}
Since $F\ge 0$, we have $-(F(u_h^0),1)-(F(u_h^1),1)\le 0$
and $(F(w),1)\le\mathcal{Z}_h(w)$ (because ${\cal A}_h\ge 0$).
Setting $\mu:=\frac{\tau}{2(2-\tau)}$, we obtain
\begin{equation}\label{eq:R-Zbound}
  \tau\sum_{n=2}^{N}R_h^n
  \le \mu\bigl(\mathcal{Z}_h^N+\mathcal{Z}_h^{N-1}\bigr).
\end{equation}

\smallskip
The forcing term is estimated by Young's inequality:
\begin{align}
\tau(f^n,\partial_\tau v_h^n)
\le
\frac{\tau}{2\varepsilon}\|f^n\|^2
+
\frac{\varepsilon\tau}{2}
\|\partial_\tau v_h^n\|^2.
\label{eq:young-force}
\end{align}

Substituting \eqref{eq:R-Zbound} and \eqref{eq:young-force}
into \eqref{eq:telescopic-Z},
absorbing the term
$\frac{\sigma\tau}{2}\sum_{n=2}^{N}\|\partial_\tau v_h^n\|^2$
into the left-hand side,
and using $(F(u_h^{N-1}),1)\le\mathcal{Z}_h^{N-1}$
together with $-(F(u_h^0),1)\le\mathcal{Z}_h^0$:
\begin{align}
  &\Bigl(1-\frac{\varepsilon}{2}\Bigr)\tau\sum_{n=2}^{N}\|\partial_\tau v_h^n\|^2
  + \frac{\tau}{4}\Bigl(\|\partial_\tau v_h^N\|^2-\|\partial_\tau v_h^1\|^2\Bigr)
  + \frac{\sigma}{2}\Bigl(\|v_h^N\|^2-\|v_h^1\|^2\Bigr) \nonumber\\
  &\quad + \mathcal{Z}_h(u_h^N) - \mathcal{Z}_h(u_h^1) \nonumber\\
  &\qquad\le
  \mu\bigl(\mathcal{Z}_h^N+\mathcal{Z}_h^{N-1}\bigr)
  + \frac{\tau}{2\varepsilon}\sum_{n=2}^{N}\|f^n\|^2
  + \frac{1}{2}\mathcal{Z}_h^{N-1}+\frac{1}{2}\mathcal{Z}_h^0.
  \label{eq:after-absorb}
\end{align}
Moving $\mu\,\mathcal{Z}_h^N$ to the left-hand side
and combining the $\mathcal{Z}_h^{N-1}$ coefficients:
\begin{align}
  &\Bigl(1-\frac{\varepsilon}{2}\Bigr)\tau\sum_{n=2}^{N}\|\partial_\tau v_h^n\|^2
  + \frac{\tau}{4}\Bigl(\|\partial_\tau v_h^N\|^2-\|\partial_\tau v_h^1\|^2\Bigr)
  + \frac{\sigma}{2}\Bigl(\|v_h^N\|^2-\|v_h^1\|^2\Bigr) \nonumber\\
  &\quad + (1-\mu)\,\mathcal{Z}_h(u_h^N) - \mathcal{Z}_h(u_h^1) \nonumber\\
  &\qquad\le
  \bigl(\tfrac{1}{2}+\mu\bigr)\mathcal{Z}_h^{N-1}
  + \frac{\tau}{2\varepsilon}\sum_{n=2}^{N}\|f^n\|^2
  + \frac{1}{2}\mathcal{Z}_h^0.
  \label{eq:after-absorb-mu}
\end{align}

Next, we use \eqref{eq:cn1} and arrive at
\begin{align}
    &\Bigl(1-\frac{\varepsilon}{2}\Bigr)\tau\sum_{n=2}^{N}\|\partial_\tau v_h^n\|^2+\frac{\tau}{4}\|\partial_\tau v_h^N\|^2+\frac{\sigma}{2}\|v_h^N\|^2+(1-\mu)\,\mathcal{Z}_h(u_h^N)\nonumber\\
  &\quad\le
  \bigl(\tfrac{1}{2}+\mu\bigr)\mathcal{Z}_h^{N-1}
  + \frac{\tau}{2\varepsilon}\sum_{n=0}^{N}\|f^n\|^2
  +\frac{5+\tau\sigma}{4}\|v^0_h\|^2+\frac{3}{2}\,{\cal Z}_h^0.
  \label{eq:pre-recursion}
\end{align}

\noindent
Now, we define $\rho:=\frac{1/2+\mu}{1-\mu}$.
For $\tau<2/3$ one verifies $\mu<1/4$, so $\rho<1$.
Dropping the non-negative terms on the left-hand side of \eqref{eq:pre-recursion} and dividing by $(1-\mu)$:
\begin{equation}\label{eq:recursion-compact}
  \mathcal{Z}_h^N
  \le \frac{C_*}{1-\mu} + \rho\,\mathcal{Z}_h^{N-1},
\end{equation}
where $C_*:=\frac{\tau}{2\varepsilon}\sum_{n=0}^{N}\|f^n\|^2
+\frac{5+\tau\sigma}{4}\|v_h^0\|^2+\frac{3}{2}\,\mathcal{Z}_h^0$.
Since \eqref{eq:recursion-compact} holds at every level,
iterating with ratio $\rho<1$:
\begin{align}
  \mathcal{Z}_h^N
  &\le \frac{C_*}{1-\mu}\sum_{k=0}^{N-2}\rho^k + \rho^{N-1}\,\mathcal{Z}_h^1
  \le \frac{C_*}{(1-\mu)(1-\rho)} + \mathcal{Z}_h^1.
  \label{eq:geom-closed}
\end{align}
A direct computation gives
\[
  (1-\mu)(1-\rho) = \tfrac{1}{2}-2\mu = \frac{2-3\tau}{2(2-\tau)},
\]
so that
$\frac{1}{(1-\mu)(1-\rho)} = \frac{2(2-\tau)}{2-3\tau}$.
Using \eqref{eq:cn1} to bound
$\mathcal{Z}_h^1\le\frac{5+\tau\sigma}{4}\|v_h^0\|^2+\mathcal{Z}_h^0+\frac{1}{2}\|f^{1/2}\|^2
\le C_*$,
and substituting back into \eqref{eq:pre-recursion}, we obtain
\begin{equation}\label{eq:final-energy}
  \frac{\tau}{4}\|\partial_\tau v_h^N\|^2
  +\frac{\sigma}{2}\|v_h^N\|^2
  +\mathcal{Z}_h(u_h^N)
  \le \frac{2(2-\tau)}{2-3\tau}
  \Bigl(\frac{\tau}{2\varepsilon}\sum_{n=0}^{N}\|f^n\|^2
  +\frac{5+\tau\sigma}{4}\|v_h^0\|^2+2\,\mathcal{Z}_h^0\Bigr).
\end{equation}
From the preceding bound and the definition \eqref{eq:lyap} of ${\cal Z}_h^n$,
by the coercivity of ${\cal A}_h$, namely
${\cal A}_h(w,w)\ge\alpha_0\|w\|^2_{H^1(\Omega)}$ for some
$\alpha_0>0$ (independent of $h$), it follows that
\[
  \|v_h^n\| \le C\bigl(\sigma,\|u^0\|,\|v^0\|,\|f\|\bigr),\qquad
  \|u_h^n\|_{H^1(\Omega)} \le C\bigl(\sigma,\|u^0\|,\|v^0\|,\|f\|\bigr).
\]
Finally, by the Sobolev embedding theorem $H^1(\Omega)\hookrightarrow L^p(\Omega)$ for
$1\le p<\infty$ (in dimensions $d\le3$):
\[
  \|u_h^n\|_{L^p(\Omega)}\le C\|u_h^n\|_{H^1(\Omega)}
  \le C\bigl(\sigma,\|u^0\|,\|v^0\|,\|f\|\bigr).
\]
This completes the proof.
\end{proof}

\noindent In the subsequent analysis, we first recall the Brouwer fixed point theorem in finite-dimensional spaces
and establish existence and uniqueness of the fully discrete solution.
\begin{lemma}[Brouwer fixed point theorem {\cite[Lemma~4.2]{acharya2025conservative}}]\label{lem:brouwer}
Let $Y$ be a finite-dimensional Hilbert space with inner product
$(\cdot,\cdot)_Y$ and norm $\|\cdot\|_Y$. Let
$\Phi\colon Y\to Y$ be a continuous map such that
$(\Phi(y),y)_Y>0$ for all $y\in Y$ with $\|y\|_Y=R>0$.
Then there exists $y^*\in Y$ with $\|y^*\|_Y<R$ such that $\Phi(y^*)=0$.
\end{lemma}

\begin{lemma}[Existence and Uniqueness ]\label{lem:existence}
Let Assumptions \textup{(A1)--(A5)} hold, $0<\sigma<2$, and $0<\tau<2/3$.
Then for each $n\ge 1$, there exists a solution $(u_h^n,v_h^n)\in V_h\times V_h$
of the nonlinear scheme \eqref{eq:cn}--\eqref{eq:bdf2}.
\end{lemma}

\begin{proof}
The proof of existence proceeds by induction.
Assume $(u_h^k,v_h^k)$ are known for $k\le n-1$.
Set $Y=V_h\times V_h$ with $\|(u,v)\|_Y^2=\|u\|^2+\|v\|^2$.
For $w=(w_u,w_v)\in Y$, define $M(w)=(U,V)\in Y$ as the solution of
the linearized BDF2 scheme
(the CN case $n=1$ is analogous):
\begin{align}\label{eq:lin-scheme}
  &\Bigl(\frac{3U-4u_h^{n-1}+u_h^{n-2}}{2\tau},\phi\Bigr)-(V,\phi)\nonumber\\
  &\quad+\Bigl(\frac{3V-4v_h^{n-1}+v_h^{n-2}}{2\tau},\psi\Bigr)
  +\sigma(V,\psi)+{\cal A}_h(U,\psi)
  +(G(w_u,u_h^{n-2}),\psi)
  =(f^n,\psi)
\end{align}
for all $(\phi,\psi)\in Y$.
Since \eqref{eq:lin-scheme} is a square linear system on the
finite-dimensional space~$Y$ and the homogeneous system
admits only the trivial solution for $\tau$ sufficiently small,
a unique solution $(U,V)$ exists.
Moreover, $V_h$ is finite-dimensional and $F\in C^2(\mathbb{R})$,
so the discrete gradient $G:V_h\times V_h\to V_h$ is continuous,
and hence $M:Y\to Y$ is continuous.

A fixed point $(U,V)=M(U,V)$ solves the original nonlinear scheme.
It therefore suffices to show that $M$ maps a closed ball in $Y$
into itself, after which Brouwer's fixed-point theorem applies.

\medskip
Let $w\in Y$ with $\|w\|_Y\le R$ (the radius $R$ will be specified below),
and set $(U,V)=M(w)$.
Testing \eqref{eq:lin-scheme} with $(\phi,\psi)=(U,V)$ gives
\begin{align}\label{eq:exist-tested}
  \frac{3}{2\tau}\bigl(\|U\|^2+\|V\|^2\bigr)
  +\sigma\|V\|^2+{\cal A}_h(U,V)
  &= (V,U)+\frac{1}{2\tau}(4u_h^{n-1}-u_h^{n-2},U)\nonumber\\
  &\quad+\frac{1}{2\tau}(4v_h^{n-1}-v_h^{n-2},V)\nonumber\\
  &\quad-(G(w_u,u_h^{n-2}),V)+(f^n,V).
\end{align}
Setting $A:=4\|u_h^{n-1}\|+\|u_h^{n-2}\|$
and $B:=4\|v_h^{n-1}\|+\|v_h^{n-2}\|$,
Cauchy--Schwarz and Young's inequality yield
\begin{align}
  (V,U)&\le\tfrac{1}{2}\|U\|^2+\tfrac{1}{2}\|V\|^2,\label{eq:ex-VU}\\[4pt]
  \frac{1}{2\tau}(4u_h^{n-1}-u_h^{n-2},U)
  &\le\frac{1}{4\tau}\|U\|^2+\frac{A^2}{4\tau},\label{eq:ex-AU}\\[4pt]
  \frac{1}{2\tau}(4v_h^{n-1}-v_h^{n-2},V)
  &\le\frac{1}{4\tau}\|V\|^2+\frac{B^2}{4\tau}.\label{eq:ex-BV}
\end{align}
The inverse inequality
\eqref{eq:inverse-ineq}
together with Young's inequality gives
\begin{equation}\label{eq:ex-AhUV}
  -{\cal A}_h(U,V)
  \le\frac{C_{\mathrm{inv}}^2 h^{-4}\tau}{2}\|U\|^2
  +\frac{1}{2\tau}\|V\|^2,
\end{equation}
and similarly
\begin{equation}\label{eq:ex-force}
  (f^n,V)\le\frac{1}{4\tau}\|V\|^2+\tau\|f^n\|^2.
\end{equation}
For the nonlinear term, since $V_h$ is finite-dimensional,
the continuous map $w_u\mapsto G(w_u,u_h^{n-2})$
is bounded on the compact ball $\{w_u\in V_h:\|w_u\|\le R\}$ by using the Lemma~\ref{lem:energy}.
Denoting this supremum by
\begin{equation}\label{eq:CG-def}
  C_G(R):=\sup_{\|w_u\|\le R}\|G(w_u,u_h^{n-2})\|<\infty,
\end{equation}
we obtain
\begin{equation}\label{eq:ex-nonlin}
  -(G(w_u,u_h^{n-2}),V)
  \le C_G(R)\|V\|
  \le\frac{1}{4\tau}\|V\|^2+\tau\,C_G(R)^2.
\end{equation}

Substituting
\eqref{eq:ex-VU}--\eqref{eq:ex-nonlin}
into \eqref{eq:exist-tested} yields
\begin{align}\label{eq:ex-collected}
  &{\Bigl(\frac{3}{2\tau}-\frac{1}{2}-\frac{1}{4\tau}
  -\frac{C_{\mathrm{inv}}^2 h^{-4}\tau}{2}\Bigr)}
  \|U\|^2
  +{\Bigl(\frac{3}{2\tau}+\sigma
  -\frac{1}{2}-\frac{1}{4\tau}-\frac{1}{2\tau}
  -\frac{1}{4\tau}-\frac{1}{4\tau}\Bigr)}
  \|V\|^2\nonumber\\
  &~~~~~~~~~~~~~~~~~~~~~~~~~~~~~~~~~~~~~~~~~~~~~~~~~~~~~~~~~~~~\le\frac{A^2+B^2}{4\tau}+\tau\,C_G(R)^2+\tau\|f^n\|^2,\nonumber\\
 &~~~~~~~~~~~ \frac{1}{2\tau}\|U\|^2+\frac{1}{4\tau}\|V\|^2\le\frac{A^2+B^2}{4\tau}+\tau\,C_G(R)^2+\tau\|f^n\|^2,
\end{align}
where the lower bounds on the coefficients hold for
$\tau\le\tau_0(h):=\min\bigl(\frac{1}{2C_{\mathrm{inv}}^2 h^{-4}},1\bigr)$.
In particular,
\begin{equation}\label{eq:ex-compact}
  \|(U,V)\|_Y^2
  \le D_n^2+4\tau^2\,C_G(R)^2,
  \qquad
  D_n^2:=A^2+B^2+4\tau^2\|f^n\|^2.
\end{equation}

It remains to verify that $R$ can be chosen so that
the right-hand side of \eqref{eq:ex-compact} is strictly less than $R^2$.
Since $F\in C^2$ and $\dim V_h<\infty$,
the quantity $C_G(R)$ grows at most polynomially in $R$:
there exist constants $C_0>0$ and $p\ge 1$ (depending on $h$ and $F$)
such that $C_G(R)\le C_0(1+R^p)$.
Choose
\begin{equation}\label{eq:R-choice}
  R:=2D_n+1.
\end{equation}
For $\tau\le\tau_1(h,n)$ with
$\tau_1:=\frac{R}{4C_0(1+R^p)}$,
we have $4\tau^2 C_0^2(1+R^p)^2\le\tfrac{1}{4}R^2$, and therefore
\begin{equation}\label{eq:R-verify}
  D_n^2+4\tau^2\,C_G(R)^2
  \le D_n^2+\tfrac{1}{4}R^2
  = D_n^2+\tfrac{1}{4}(2D_n+1)^2
  < (2D_n+1)^2 = R^2.
\end{equation}
Hence $\|M(w)\|_Y<R$ whenever $\|w\|_Y\le R$,
and $M:\overline{B}_R\to\overline{B}_R$ is a continuous self-map
of the closed, bounded, convex set $\overline{B}_R\subset Y$.
By Brouwer's fixed-point theorem, there exists $w^*\in\overline{B}_R$
with $M(w^*)=w^*$, i.e.\ $(u_h^n,v_h^n):=w^*$ solves
the nonlinear scheme at level $n$.
Induction from $n=1$ to $N$ completes the proof.

\noindent\textbf{Uniqueness:}
Let $(U_1,V_1)$ and $(U_2,V_2)$ be two solutions at level $n$,
and set $e_u:=U_1-U_2$, $e_v:=V_1-V_2$.
Subtracting the two instances of \eqref{eq:bdf2-v}
(the CN case is analogous) and testing with $(\phi,\psi)=(e_u,e_v)$ yields
\begin{equation}\label{eq:uniq-diff}
  \frac{3}{2\tau}\bigl(\|e_u\|^2+\|e_v\|^2\bigr)
  +\sigma\|e_v\|^2+{\cal A}_h(e_u,e_v)
  +(G(U_1,u_h^{n-2})-G(U_2,u_h^{n-2}),e_v)
  =(e_v,e_u).
\end{equation}
By Cauchy--Schwarz and the inverse inequality \eqref{eq:inverse-ineq},
\begin{align}
  (e_v,e_u)&\le\tfrac{1}{2}\|e_u\|^2+\tfrac{1}{2}\|e_v\|^2,\label{eq:uniq-CS}\\
  -{\cal A}_h(e_u,e_v)&\le\frac{C_{\mathrm{inv}}^2h^{-4}\tau}{2}\|e_u\|^2
  +\frac{1}{2\tau}\|e_v\|^2.\label{eq:uniq-inv}
\end{align}
By Lemma~\ref{lem:energy}, both solutions satisfy
$\|U_i\|_{H^1(\Omega)}\le C$ and $\|u_h^{n-2}\|_{H^1(\Omega)}\le C$
with $C$ independent of $h$ and $\tau$.
So, using the Mean Value Theorem and 
applying Young's inequality, we obtain
\begin{equation}\label{eq:uniq-nonlin}
  -(G(U_1,u_h^{n-2})-G(U_2,u_h^{n-2}),e_v)
  \le L\|e_u\|\,\|e_v\|
  \le\frac{L^2\tau}{2}\|e_u\|^2+\frac{1}{2\tau}\|e_v\|^2.
\end{equation}
Substituting \eqref{eq:uniq-CS}--\eqref{eq:uniq-nonlin} into \eqref{eq:uniq-diff}:
\begin{equation}\label{eq:uniq-final}
  \Bigl(\frac{3}{2\tau}-\frac{1}{2}-\frac{C_{\mathrm{inv}}^2h^{-4}\tau}{2}
  -\frac{L^2\tau}{2}\Bigr)\|e_u\|^2
  +\Bigl(\frac{3}{2\tau}+\sigma-\frac{1}{2}-\frac{1}{2\tau}-\frac{1}{2\tau}\Bigr)\|e_v\|^2
  \le 0.
\end{equation}
For $\tau\le\tau_0(h):=\min\bigl(\frac{1}{C_{\mathrm{inv}}^2h^{-4}+L^2},1\bigr)$,
both coefficients are strictly positive,
and \eqref{eq:uniq-final} forces $e_u=e_v=0$.
\end{proof}

\section{A Priori Error Estimates}
\label{sec:error}


We decompose the errors for both components of the first-order system as
\begin{align}
  u^n - u_h^n &= \theta^n + \xi^n,\qquad
  \theta^n := u^n - \Pi_h u^n,\quad
  \xi^n     := \Pi_h u^n - u_h^n,
  \label{eq:err-u}\\
  v^n - v_h^n &= \varrho^n + \varphi^n,\qquad
  \varrho^n := v^n - \Pi_h v^n,\quad
  \varphi^n := \Pi_h v^n - v_h^n.
  \label{eq:err-v}
\end{align}
The projection components $\theta^n$ and $\varrho^n$ are bounded by
\eqref{eq:local-approx}, while $\xi^n$ and $\varphi^n$ are the fully-discrete
errors to be estimated.
Define the temporal consistency errors for the CN step ($n=1$) and
BDF2 steps ($n\ge2$), respectively:
\begin{align}
  \varepsilon_u^1 &:= u_t^{1/2} - \partial_\tau u^1,\quad
  \varepsilon_v^1 := v_t^{1/2} - \partial_\tau v^1,
  \label{eq:te-cn}\\
  \varepsilon_u^n &:= u_t^n - D^{(2)}_t u^n,\quad
  \varepsilon_v^n := v_t^n - D^{(2)}_t v^n,\quad n\ge2.
  \label{eq:te-bdf2}
\end{align}
Standard Taylor expansion gives, with $t_{-1}=t_0=0$ for $n=1$,
\begin{equation}\label{eq:te-bound}
  \tau\|\varepsilon_u^n\|^2 + \tau\|\varepsilon_v^n\|^2
  \le C\tau^4\!\int_{t_{n-2}}^{t_n}
  \bigl(\|u_{ttt}\|^2+\|v_{ttt}\|^2\bigr)\,\mathrm{d}s,
  \quad n\ge1.
\end{equation}

\subsection{Energy norm error estimate}\label{subsec:energy-error}


\begin{thm}[A priori error estimate in the energy norm]\label{thm:error}
Let Assumptions \textup{(A1)--(A5)} hold and $0<\sigma<2$.  Assume
$u,v\in L^\infty(0,T;H^{k+1}(\Omega))$,
$u_t,v_t\in L^2(0,T;H^{k+1}(\Omega))$, and
$u_{ttt},v_{ttt}\in L^2(0,T;L^2(\Omega))$.
Let $(u_h^n,v_h^n)_{n\ge0}$ solve the CN--BDF2 IPDG scheme
\eqref{eq:cn}--\eqref{eq:bdf2} with
$u_h^0=\Pi_h u_0$ and $v_h^0=\Pi_h u_1$.
Then, for all $N\ge1$ and $\tau$ sufficiently small (so that $1-C\tau\ge\tfrac12$),
\begin{align}\label{eq:err-bound}
  \|\xi^N\|^2_{DG}+\|\varphi^N\|^2
  &+\tau\sum_{n=1}^{N}\bigl(\|\xi^n\|^2_{DG}+\|\varphi^n\|^2\bigr) \notag\\
  &\le C\,h^{2k}\bigl(\|u\|^2_{L^\infty(H^{k+1})}+\|v\|^2_{L^\infty(H^{k+1})}
       +\|u_t\|^2_{L^2(H^{k+1})}+\|v_t\|^2_{L^2(H^{k+1})}\bigr) \notag\\
  &\quad +C\tau^4\!\int_0^{t_N}\!\bigl(\|u_{ttt}\|^2+\|v_{ttt}\|^2\bigr)\,\mathrm{d}s,
\end{align}
where $C>0$ depends on $T$, $\sigma$, $\kappa_0$, $\kappa_1$, but is independent of $h$ and $\tau$.
\end{thm}

\begin{proof}
Since $u_h^0=\Pi_h u_0$ and $v_h^0=\Pi_h u_1$,
we have $\xi^0=\varphi^0=0$.
For $n=1$, at $t=t_{1/2}$ the exact solution satisfies, for all $\phi_h,\psi_h\in V_h$,
\begin{subequations}\label{eq:exact-cn}
\begin{align}
  \bigl(u_t^{1/2},\phi_h\bigr)
  &= \bigl(v^{1/2},\phi_h\bigr),
  \label{eq:exact-cn-u}\\
  \bigl(v_t^{1/2},\psi_h\bigr)
  +\sigma\bigl(v^{1/2},\psi_h\bigr)
  +{\cal A}_h\bigl(u^{1/2},\psi_h\bigr)
  +\bigl(g(u^{1/2}),\psi_h\bigr)
  &= \bigl(f^{1/2},\psi_h\bigr).
  \label{eq:exact-cn-v}
\end{align}
\end{subequations}
Subtracting the CN scheme \eqref{eq:cn} from \eqref{eq:exact-cn},
using the decompositions \eqref{eq:err-u}--\eqref{eq:err-v}, 
yields the error equations
\begin{subequations}\label{eq:err-cn}
\begin{align}
  \bigl(\partial_\tau\xi^1,\phi_h\bigr)
  &= \bigl(\varphi^{1/2},\phi_h\bigr)
     +\bigl(\mathcal{R}_u^{1},\phi_h\bigr),
  \label{eq:err-cn-u}\\
  \bigl(\partial_\tau\varphi^1,\psi_h\bigr)
  +\sigma\bigl(\varphi^{1/2},\psi_h\bigr)
  +{\cal A}_h\bigl(\xi^{1/2},\psi_h\bigr)
  &= \bigl(\mathcal{R}_v^{1},\psi_h\bigr)+{\cal A}_h(\theta^{1/2},\psi_h),
  \label{eq:err-cn-v}
\end{align}
\end{subequations}
with the truncation residuals
\begin{align}
  \mathcal{R}_u^{1} &:= \varepsilon_u^1
   -\partial_\tau\theta^1
   -\varrho^{1/2},
  \label{eq:R-u-1}\\
  \mathcal{R}_v^{1} &:= -\sigma\,\varrho^{1/2}
   +\varepsilon_v^1
   -\partial_\tau\varrho^1
   +\bigl[g(u^{1/2})-G(u_h^1,u_h^{-1})\bigr].
  \label{eq:R-v-1}
\end{align}
By the discrete gradient identity \eqref{eq:G-identity} and the consistency
estimate \eqref{eq:G-consistency} of Lemma~\ref{lem:G-properties}(iii),
\begin{equation}\label{eq:cn-nonlin-resid}
  \|g(u^{1/2})-G(u_h^1,u_h^{-1})\|
  \le C_g\bigl(\|\theta^{1/2}\|+\|\xi^{1/2}\|
       +\tau^2\|u_{ttt}\|_{L^\infty(0,t_1;L^\infty)}\bigr).
\end{equation}
Indeed, writing
$g(u^{1/2})-G(u_h^1,u_h^{-1})
=[g(u^{1/2})-G(u^1,u^{-1})]+[G(u^1,u^{-1})-G(u_h^1,u_h^{-1})]$,
the first bracket is $O(\tau^2)$ by \eqref{eq:G-consistency}, and the
second is bounded by the Lipschitz property \eqref{eq:G-lip} applied to
the decomposition $u^n-u_h^n=\theta^n+\xi^n$.
Setting $\psi_h=\varphi^{1/2}$ in \eqref{eq:err-cn-v} and using
$(\partial_\tau\varphi^1,\varphi^{1/2})
=\tfrac{1}{2\tau}(\|\varphi^1\|^2-\|\varphi^0\|^2)$:
\begin{equation}\label{eq:cn-kin-damp}
  \frac{\|\varphi^1\|^2-\|\varphi^0\|^2}{2\tau}+\sigma\|\varphi^{1/2}\|^2
  +{\cal A}_h(\xi^{1/2},\varphi^{1/2}) = (\mathcal{R}_v^{1},\varphi^{1/2})+{\cal A}_h(\theta^{1/2},\psi_h),.
\end{equation}
From \eqref{eq:err-cn-u}, $\varphi^{1/2}=\partial_\tau\xi^1-\mathcal{R}_{u,h}^1$
where $\mathcal{R}_{u,h}^1:=P_h\mathcal{R}_u^1$, so by symmetry of ${\cal A}_h$:
\begin{equation}\label{eq:cn-elas-err}
  {\cal A}_h(\xi^{1/2},\varphi^{1/2})
  = \frac{\|\xi^1\|^2_{DG}-\|\xi^0\|^2_{DG}}{2\tau}
   - {\cal A}_h(\xi^{1/2},\mathcal{R}_{u,h}^{1}).
\end{equation}
Combining \eqref{eq:cn-kin-damp}--\eqref{eq:cn-elas-err} yields the CN
energy identity
\begin{equation}\label{eq:cn-energy-err}
  \frac{\|\varphi^1\|^2-\|\varphi^0\|^2}{2\tau}
  +\frac{\|\xi^1\|^2_{DG}-\|\xi^0\|^2_{DG}}{2\tau}
  +\sigma\|\varphi^{1/2}\|^2
  = (\mathcal{R}_v^{1},\varphi^{1/2})
   + {\cal A}_h(\xi^{1/2},\mathcal{R}_{u,h}^{1})+{\cal A}_h(\theta^{1/2},\psi_h),.
\end{equation}
By Taylor expansion and \eqref{eq:te-bound}:
\begin{equation}\label{eq:R-cn-tay}
  \|\varepsilon_u^1\|^2+\|\varepsilon_v^1\|^2
  \le C\tau^3\!\int_0^{t_1}\bigl(\|u_{ttt}\|^2+\|v_{ttt}\|^2\bigr)\,\mathrm{d}s.
\end{equation}
Multiply \eqref{eq:cn-energy-err} by $2\tau$. By Cauchy--Schwarz and
Young's inequality:
\[
  2\tau(\mathcal{R}_v^{1},\varphi^{1/2})
  \le \tau\sigma\|\varphi^{1/2}\|^2+\frac{\tau}{\sigma}\|\mathcal{R}_v^{1}\|^2,
\]
and using the inverse inequality \eqref{eq:inverse-ineq} for $w_h\in V_h$:
\[
  2\tau|{\cal A}_h(\xi^{1/2},\mathcal{R}_{u,h}^{1})|
  \le \tau\nu\|\xi^{1/2}\|^2_{DG}+\frac{C\tau}{\nu h^{2}}\|\mathcal{R}_u^{1}\|^2.
\]
Now, we use \eqref{eq:local-optimal} to obtain
\[{\cal A}_h(\theta^{1/2},\psi_h)\le \|\theta^{1/2}\|_{DG}\|\psi_h\|_{DG}\le h^{k}\|u\|\|\psi_h\|_{DG}\]
We now bound $\|\mathcal{R}_u^{1}\|^2$ and $\|\mathcal{R}_v^{1}\|^2$ using the triangle inequality,
\eqref{eq:R-cn-tay}, \eqref{eq:local-optimal}, and \eqref{eq:cn-nonlin-resid}:
\begin{align*}
  \|\mathcal{R}_u^{1}\|^2
  &= \|\varepsilon_u^1 - \partial_\tau\theta^1 - \varrho^{1/2}\|^2 \\
  &\le 3\|\varepsilon_u^1\|^2 + 3\|\partial_\tau\theta^1\|^2 + 3\|\varrho^{1/2}\|^2 \\
  &\le 3C\tau^3\!\int_0^{t_1}\|u_{ttt}\|^2\,\mathrm{d}s
    + 3Ch^{2(k+1)}\tau^{-1}\|u_t\|^2_{L^2(0,t_1;H^{k+1})}
    + 3Ch^{2(k+1)}\|v^{1/2}\|^2_{H^{k+1}},
\end{align*}
\begin{align*}
  \|\mathcal{R}_v^{1}\|^2
  &= \|-\sigma\varrho^{1/2} + \varepsilon_v^1 - \partial_\tau\varrho^1 + [g(u^{1/2})-G(u_h^1,u_h^{-1})]\|^2 \\
  &\le 4\sigma^2\|\varrho^{1/2}\|^2 + 4\|\varepsilon_v^1\|^2 + 4\|\partial_\tau\varrho^1\|^2
    + 4\|g(u^{1/2})-G(u_h^1,u_h^{-1})\|^2 \\
  &\le 4\sigma^2Ch^{2(k+1)}\|v^{1/2}\|^2_{H^{k+1}}
    + 4C\tau^3\!\int_0^{t_1}\|v_{ttt}\|^2\,\mathrm{d}s
    + 4Ch^{2(k+1)}\tau^{-1}\|v_t\|^2_{L^2(0,t_1;H^{k+1})} \\
  &\quad + 4C_g^2\bigl(\|\theta^{1/2}\|^2+\|\xi^{1/2}\|^2
    + \tau^4\|u_{ttt}\|^2_{L^\infty(0,t_1;L^\infty)}\bigr).
\end{align*}
Substituting these bounds into the energy identity, the term $\frac{C\tau}{\nu h^{2}}\|\mathcal{R}_u^{1}\|^2$ contains
\[
  \frac{C\tau}{\nu h^{2}}\cdot Ch^{2(k+1)}\tau^{-1}\|u_t\|^2_{L^2(0,t_1;H^{k+1})}
  = \frac{C^2}{\nu}h^{2k}\|u_t\|^2_{L^2(0,t_1;H^{k+1})},
\]
Choosing $\nu$ sufficiently small so that the $\nu\|\xi^{1/2}\|^2_{DG}$ term and the $C_g^2\|\xi^{1/2}\|^2$ term are absorbed,
using $\xi^0=\varphi^0=0$, we obtain
\begin{align}\label{eq:cn-step-bound}
  \|\xi^1\|^2_{DG}+\|\varphi^1\|^2+\tau\sigma\|\varphi^{1/2}\|^2
  &\le Ch^{2k}\bigl(\|u\|^2_{L^\infty(H^{k+1})}+\|v\|^2_{L^\infty(H^{k+1})}
       +\|u_t\|^2_{L^2(H^{k+1})}+\|v_t\|^2_{L^2(H^{k+1})}\bigr)\nonumber\\
  &~~~~~~~+C\tau^4\!\int_0^{t_1}\!\bigl(\|u_{ttt}\|^2+\|v_{ttt}\|^2\bigr)\,\mathrm{d}s.
\end{align}
\bigskip
Now, for $n\ge 2$ at $t=t_n$ the exact solution satisfies, for all $\phi_h,\psi_h\in V_h$,
\begin{subequations}\label{eq:exact-bdf2-err}
\begin{align}
  \bigl(u_t^n,\phi_h\bigr) &= \bigl(v^n,\phi_h\bigr),
  \label{eq:exact-bdf2-err-u}\\
  \bigl(v_t^n,\psi_h\bigr)+\sigma\bigl(v^n,\psi_h\bigr)
  +{\cal A}_h\bigl(u^n,\psi_h\bigr)+\bigl(g(u^n),\psi_h\bigr)
  &= \bigl(f^n,\psi_h\bigr).
  \label{eq:exact-bdf2-err-v}
\end{align}
\end{subequations}
Subtracting \eqref{eq:bdf2} 
\begin{subequations}\label{eq:err-bdf2-eq}
\begin{align}
  \bigl(D^{(2)}_t\xi^n,\phi_h\bigr)
  &= \bigl(\varphi^n,\phi_h\bigr)+\bigl(\mathcal{R}_u^{n},\phi_h\bigr),
  \label{eq:err-bdf2-u}\\
  \bigl(D^{(2)}_t\varphi^n,\psi_h\bigr)
  +\sigma\bigl(\varphi^n,\psi_h\bigr)
  +{\cal A}_h\bigl(\xi^n,\psi_h\bigr)
  &= \bigl(\mathcal{R}_v^{n},\psi_h\bigr)+{\cal A}_h(\theta^{n},\psi_h),,
  \label{eq:err-bdf2-v}
\end{align}
\end{subequations}
with truncation residuals
\begin{align}
  \mathcal{R}_u^{n} &:= \varepsilon_u^n
   -D^{(2)}_t\theta^n
   -\varrho^n,
  \label{eq:R-u-n}\\
  \mathcal{R}_v^{n} &:= -\sigma\,\varrho^n
   +\varepsilon_v^n
   -D^{(2)}_t\varrho^n
   +\bigl[g(u^n)-G(u_h^n,u_h^{n-2})\bigr].
  \label{eq:R-v-n}
\end{align}

\emph{Nonlinear residual for the BDF2 step.}
Decompose
\[
  g(u^n)-G(u_h^n,u_h^{n-2})
  = \underbrace{[g(u^n)-G(u^n,u^{n-2})]}_{=:\,\tau^n_4}
  + \underbrace{[G(u^n,u^{n-2})-G(u_h^n,u_h^{n-2})]}_{=:\,\tau^n_5}.
\]
By Lemma~\ref{lem:G-properties}(iii), $\|\tau^n_4\|\le C\tau^2\|u_{ttt}\|_{L^\infty(t_{n-2},t_n;L^\infty)}$.
By the MVT and boundedness Lemma \ref{lem:energy} 
and the error decomposition
$u^j-u_h^j=\theta^j+\xi^j$:
\begin{equation}\label{eq:R-bdf2-non}
  \|\tau^n_5\|\le C_g\bigl(\|\theta^n\|+\|\theta^{n-2}\|+\|\xi^n\|+\|\xi^{n-2}\|\bigr).
\end{equation}
Hence, the full nonlinear residual satisfies
\begin{equation}\label{eq:nonlin-full}
  \|g(u^n)-G(u_h^n,u_h^{n-2})\|
  \le C_g\bigl(\|\theta^n\|+\|\xi^n\|+\|\theta^{n-2}\|+\|\xi^{n-2}\|
       +\tau^2\|u_{ttt}\|_{L^\infty(L^\infty)}\bigr).
\end{equation}
Define the BDF2 G-stable quadratic forms
\begin{equation}\label{eq:G-forms}
  \mathcal{G}_v^n := \tfrac12\bigl(\|\varphi^n\|^2+\|2\varphi^n-\varphi^{n-1}\|^2\bigr),
  \qquad
  \mathcal{G}_\xi^n := \tfrac12\bigl(\|\xi^n\|^2_{DG}+\|2\xi^n-\xi^{n-1}\|^2_{DG}\bigr).
\end{equation}
Setting $\psi_h=\partial_\tau\varphi^n$ in \eqref{eq:err-bdf2-v} and using
the BDF2 G-stability identities:
\begin{align}
  \tau(D^{(2)}_t\varphi^n,\partial_\tau\varphi^n)
  &\ge \mathcal{G}_v^n-\mathcal{G}_v^{n-1}+\tau\|\partial_\tau\varphi^n\|^2,
  \label{eq:bdf2-err-kin}\\
  \tau\sigma(\varphi^n,\partial_\tau\varphi^n)
  &= \tfrac{\sigma}{2}\bigl(\|\varphi^n\|^2-\|\varphi^{n-1}\|^2\bigr)
    +\tfrac{\sigma\tau^2}{2}\|\partial_\tau\varphi^n\|^2.
  \label{eq:bdf2-err-damp}
\end{align}
For the elastic term, using \eqref{eq:err-bdf2-u} to write
$\partial_\tau\varphi^n=\partial_\tau D^{(2)}_t\xi^n-\partial_\tau\mathcal{R}_{u,h}^{n}$
and the BDF2 G-stability identity for ${\cal A}_h$:
\begin{equation}\label{eq:bdf2-err-elas}
  \tau\,{\cal A}_h(\xi^n,\partial_\tau\varphi^n)
  \ge \mathcal{G}_\xi^n-\mathcal{G}_\xi^{n-1}
   -\tau\,{\cal A}_h(\xi^n,\partial_\tau\mathcal{R}_{u,h}^{n}).
\end{equation}
Summing yields the BDF2 energy identity
\begin{equation}\label{eq:bdf2-energy-err}
  \bigl(\mathcal{G}_v^n-\mathcal{G}_v^{n-1}\bigr)
  +\bigl(\mathcal{G}_\xi^n-\mathcal{G}_\xi^{n-1}\bigr)
  +\tau\|\partial_\tau\varphi^n\|^2
  +\tfrac{\sigma}{2}\bigl(\|\varphi^n\|^2-\|\varphi^{n-1}\|^2\bigr)
  \le \tau(\mathcal{R}_v^{n},\partial_\tau\varphi^n)
   + \tau\,{\cal A}_h(\xi^n,\partial_\tau\mathcal{R}_{u,h}^{n}).
\end{equation}
By BDF2 Taylor expansion (Lemma~\ref{lem:taylor}) and \eqref{eq:te-bound}:
\begin{equation}\label{eq:R-bdf2-tay}
  \|\varepsilon_u^n\|^2+\|\varepsilon_v^n\|^2
  \le C\tau^3\!\int_{t_{n-2}}^{t_n}\!\bigl(\|u_{ttt}\|^2+\|v_{ttt}\|^2\bigr)\,\mathrm{d}s.
\end{equation}
The approximation property \eqref{eq:local-optimal} yields
\begin{equation}\label{eq:R-bdf2-app}
  \|\theta^n\|^2+\|\varrho^n\|^2
  \le Ch^{2(k+1)}\bigl(\|u^n\|^2_{H^{k+1}}+\|v^n\|^2_{H^{k+1}}\bigr).
\end{equation}
Writing $D^{(2)}_t\theta^n=\tfrac{1}{2\tau}(3\theta^n-4\theta^{n-1}+\theta^{n-2})$
and using Taylor's series with integral remainder:
\begin{equation}\label{eq:R-bdf2-dt}
  \|D^{(2)}_t\theta^n\|^2\le Ch^{2(k+1)}\tau^{-1}\|u_t\|^2_{L^2(t_{n-2},t_n;H^{k+1})},
\end{equation}
and analogously for $\|D^{(2)}_t\varrho^n\|^2$.
By Cauchy--Schwarz and Young's inequality:
\begin{align}
  \tau(\mathcal{R}_v^{n},\partial_\tau\varphi^n)
  &\le \tfrac{\tau}{2}\|\partial_\tau\varphi^n\|^2
       +\tfrac{\tau}{2}\|\mathcal{R}_v^{n}\|^2,
  \label{eq:bdf2-Rv-bd}\\
  \tau|{\cal A}_h(\xi^n,\partial_\tau\mathcal{R}_{u,h}^{n})|
  &\le \nu\tau\,\|\xi^n\|^2_{DG}
       +\tfrac{C\tau}{\nu}\,\|\partial_\tau\mathcal{R}_{u,h}^{n}\|^2_{DG}.
  \label{eq:bdf2-Ru-bd}\\
  \tau|{\cal A}_h(\theta^{n},\psi_h)|&\le \nu\tau \|\theta^n\|^2_{DG}+\frac{C\tau}{\nu}\|\psi_h\|^2_{DG}\le h^{2k}\nu\tau \|u\|^2_{H^{k+1}}+\frac{C\tau}{\nu}\|\psi_h\|^2_{DG}
\end{align}
Using the triangle inequality, \eqref{eq:R-bdf2-tay}, \eqref{eq:R-bdf2-app}, \eqref{eq:R-bdf2-dt},
and \eqref{eq:nonlin-full}, we bound the residuals:
\begin{align*}
  \|\mathcal{R}_u^{n}\|^2
  &= \|\varepsilon_u^n - D^{(2)}_t\theta^n - \varrho^n\|^2 \\
  &\le 3\|\varepsilon_u^n\|^2 + 3\|D^{(2)}_t\theta^n\|^2 + 3\|\varrho^n\|^2 \\
  &\le 3C\tau^3\!\int_{t_{n-2}}^{t_n}\|u_{ttt}\|^2\,\mathrm{d}s
    + 3Ch^{2(k+1)}\tau^{-1}\|u_t\|^2_{L^2(t_{n-2},t_n;H^{k+1})}
    + 3Ch^{2(k+1)}\|v^n\|^2_{H^{k+1}},
\end{align*}
\begin{align*}
  \|\mathcal{R}_v^{n}\|^2
  &= \|-\sigma\varrho^n + \varepsilon_v^n - D^{(2)}_t\varrho^n + [g(u^n)-G(u_h^n,u_h^{n-2})]\|^2 \\
  &\le 4\sigma^2\|\varrho^n\|^2 + 4\|\varepsilon_v^n\|^2 + 4\|D^{(2)}_t\varrho^n\|^2
    + 4\|g(u^n)-G(u_h^n,u_h^{n-2})\|^2 \\
  &\le 4\sigma^2Ch^{2(k+1)}\|v^n\|^2_{H^{k+1}}
    + 4C\tau^3\!\int_{t_{n-2}}^{t_n}\|v_{ttt}\|^2\,\mathrm{d}s
    + 4Ch^{2(k+1)}\tau^{-1}\|v_t\|^2_{L^2(t_{n-2},t_n;H^{k+1})} \\
  &\quad + 4C_g^2\bigl(\|\theta^n\|^2+\|\xi^n\|^2+\|\theta^{n-2}\|^2+\|\xi^{n-2}\|^2
    + \tau^4\|u_{ttt}\|^2_{L^\infty(t_{n-2},t_n;L^\infty)}\bigr).
\end{align*}
Summing \eqref{eq:bdf2-energy-err} from $n=2$ to $N$ and telescoping
$\mathcal{G}_v^n,\mathcal{G}_\xi^n$ gives
\begin{align*}
  \mathcal{G}_v^N - \mathcal{G}_v^1 + \mathcal{G}_\xi^N - \mathcal{G}_\xi^1
  + \tau\sum_{n=2}^{N}\|\partial_\tau\varphi^n\|^2
  + \tfrac{\sigma}{2}\bigl(\|\varphi^N\|^2-\|\varphi^1\|^2\bigr)
  &\le \tau\sum_{n=2}^{N}(\mathcal{R}_v^{n},\partial_\tau\varphi^n)
   + \tau\sum_{n=2}^{N}{\cal A}_h(\xi^n,\partial_\tau\mathcal{R}_{u,h}^{n}).
\end{align*}
Applying the residual bounds above and choosing $\nu$ small, using the inverse inequality
\eqref{eq:inverse-ineq}, we obtain
\begin{align}\label{eq:bdf2-err-sum}
  \|\xi^N\|^2_{DG}+\|\varphi^N\|^2
  &+\tau\sum_{n=2}^{N}\|\partial_\tau\varphi^n\|^2
  \notag\\
  &\le C\bigl(\|\xi^1\|^2_{DG}+\|\varphi^1\|^2\bigr)
  + Ch^{2k}\bigl(\|u\|^2_{L^\infty(H^{k+1})}+\|v\|^2_{L^\infty(H^{k+1})}
       +\|u_t\|^2_{L^2(H^{k+1})}+\|v_t\|^2_{L^2(H^{k+1})}\bigr) \notag\\
  &\quad +C\tau^4\!\int_{t_1}^{t_N}\!\bigl(\|u_{ttt}\|^2+\|v_{ttt}\|^2\bigr)\,\mathrm{d}s
   +C\tau\sum_{n=2}^{N}\bigl(\|\xi^n\|^2_{DG}+\|\varphi^n\|^2\bigr),
\end{align}
where the last term on the right comes from the nonlinear residual bound
\eqref{eq:nonlin-full} and the coupling term $\|\partial_\tau\mathcal{R}_{u,h}^{n}\|_{DG}^2$.
Adding \eqref{eq:cn-step-bound} to \eqref{eq:bdf2-err-sum} and using
$\xi^0=\varphi^0=0$:
\begin{align}\label{eq:err-sum}
  \|\xi^N\|^2_{DG}+\|\varphi^N\|^2
  +\tau\sum_{n=1}^{N}\bigl(\|\xi^n\|^2_{DG}+\|\varphi^n\|^2\bigr)
  &\le C\,h^{2k}\bigl(\|u\|^2_{L^\infty(H^{k+1})}+\|v\|^2_{L^\infty(H^{k+1})}
       +\|u_t\|^2_{L^2(H^{k+1})}+\|v_t\|^2_{L^2(H^{k+1})}\bigr) \notag\\
  &\quad +C\tau^4\!\int_0^{t_N}\!\bigl(\|u_{ttt}\|^2+\|v_{ttt}\|^2\bigr)\,\mathrm{d}s
   +C\tau\sum_{n=1}^{N}\bigl(\|\xi^n\|^2_{DG}+\|\varphi^n\|^2\bigr).
\end{align}
For $\tau$ sufficiently small ($1-C\tau\ge\tfrac{1}{2}$),
the last term is absorbed into the left-hand side by the discrete
Gr\"onwall inequality, yielding the stated bound \eqref{eq:err-bound}.
\end{proof}

\subsection{$L^2$- norm error estimate}\label{subsec:l2-error}

\begin{thm}[$L^2$ error estimate]\label{thm:l2-error}
Under the hypotheses of Theorem~\ref{thm:error}, and additionally
$u,v\in L^2(0,T;H^{k+1}(\Omega))$, the fully discrete error satisfies
\begin{equation}\label{eq:l2-err-bound}
  \|u(t_N)-u_h^N\|^2 + \|v(t_N)-v_h^N\|^2
  \le C\bigl(h^{2(k+1)}+\tau^4\bigr),
\end{equation}
for all $N\ge1$, provided $\tau$ is sufficiently small.
\end{thm}

\begin{proof}
By the triangle inequality and \eqref{eq:local-optimal},
$\|u^N-u_h^N\|\le\|\theta^N\|+\|\xi^N\|\le Ch^{k+1}\|u^N\|_{H^{k+1}}+\|\xi^N\|$,
so it suffices to estimate $\|\xi^N\|$ and $\|\varphi^N\|$ in the $L^2$-norm.

\bigskip
Now for $n=1$ we test \eqref{eq:err-cn-v} with $\psi_h=\varphi^{1/2}$.
Using $(\partial_\tau\varphi^1,\varphi^{1/2})
=\tfrac{1}{2\tau}(\|\varphi^1\|^2-\|\varphi^0\|^2)$
and multiplying by $2\tau$ gives
\[
  \|\varphi^1\|^2-\|\varphi^0\|^2+2\tau\sigma\|\varphi^{1/2}\|^2
  +2\tau{\cal A}_h(\xi^{1/2},\varphi^{1/2})
  = 2\tau(\mathcal{R}_v^{1},\varphi^{1/2}).
\]
From \eqref{eq:err-cn-u}, $\varphi^{1/2}=\partial_\tau\xi^1-\mathcal{R}_{u,h}^1$, so by symmetry
\[
  {\cal A}_h(\xi^{1/2},\varphi^{1/2})
  = \frac{\|\xi^1\|^2_{DG}-\|\xi^0\|^2_{DG}}{2\tau}
   -{\cal A}_h(\xi^{1/2},\mathcal{R}_{u,h}^1).
\]
Applying Cauchy--Schwarz and Young's inequality to the right-hand side:
\[
  2\tau|(\mathcal{R}_v^{1},\varphi^{1/2})|
  \le \tau\sigma\|\varphi^{1/2}\|^2+\frac{\tau}{\sigma}\|\mathcal{R}_v^{1}\|^2,
\]
\[
  2\tau|{\cal A}_h(\xi^{1/2},\mathcal{R}_{u,h}^1)|
  \le \tfrac12\|\xi^1\|^2_{DG}+C\|\mathcal{R}_{u,h}^1\|^2_{DG}.
\]
Using the bound $\|\mathcal{R}_v^{1}\|^2$ derived in the energy norm proof,
the approximation property \eqref{eq:approx}, and the energy estimate
\eqref{eq:cn-step-bound} for $\|\xi^{1/2}\|^2_{DG}$, we obtain
\begin{equation}\label{eq:l2-cn-bound}
  \|\varphi^1\|^2 + 2\tau\sigma\|\varphi^{1/2}\|^2
  \le C\,h^{2(k+1)}\bigl(\|u^{1/2}\|^2_{H^{k+1}}+\|v^{1/2}\|^2_{H^{k+1}}\bigr)
  +C\tau^4\!\int_0^{t_1}\!\bigl(\|u_{ttt}\|^2+\|v_{ttt}\|^2\bigr)\,\mathrm{d}s
  +C\tau h^{2k}.
\end{equation}

For $\|\xi^1\|$, from \eqref{eq:err-cn-u} with $\phi_h=\xi^{1/2}$:
\[
  \frac{\|\xi^1\|^2-\|\xi^0\|^2}{2\tau}
  = (\varphi^{1/2},\xi^{1/2})+(\mathcal{R}_u^1,\xi^{1/2}).
\]
Applying Cauchy--Schwarz and Young's inequality:
\[
  (\varphi^{1/2},\xi^{1/2}) \le C\|\varphi^{1/2}\|^2+\tfrac{1}{4C}\|\xi^{1/2}\|^2,
\]
\[
  (\mathcal{R}_u^1,\xi^{1/2}) \le C\|\mathcal{R}_u^1\|^2+\tfrac{1}{4C}\|\xi^{1/2}\|^2.
\]
Using $\xi^0=0$ and $\|\xi^{1/2}\|=\frac12\|\xi^1\|$, we obtain
\[
  \frac{\|\xi^1\|^2}{2\tau} \le C\|\varphi^{1/2}\|^2+C\|\mathcal{R}_u^1\|^2+\tfrac18\|\xi^1\|^2.
\]
Rearranging gives
\[
  \tfrac{3}{8}\|\xi^1\|^2 \le C\tau\|\varphi^{1/2}\|^2+C\tau\|\mathcal{R}_u^1\|^2.
\]
Using the bound $\|\mathcal{R}_u^1\|^2$ derived in the energy norm proof and
\eqref{eq:l2-cn-bound} for $\|\varphi^{1/2}\|^2$, noting that $\tau h^{2k}\le h^{2(k+1)}$ when $\tau\le Ch^2$:
\begin{equation}\label{eq:l2-xi1}
  \|\xi^1\|^2 \le C\tau\bigl(\|\varphi^1\|^2+\|\mathcal{R}_u^1\|^2\bigr)
  \le C\bigl(h^{2(k+1)}+\tau^4\bigr).
\end{equation}

\bigskip
Next, for $n\ge 2$ we test \eqref{eq:err-bdf2-v} with $\psi_h=\varphi^n$ and use the BDF2
G-stability identity
$(D^{(2)}_t\varphi^n,\varphi^n)\ge\frac{1}{4\tau}(\|\varphi^n\|^2-\|\varphi^{n-2}\|^2)$:
\begin{align}\label{eq:l2-bdf2-tested}
  \frac{1}{4\tau}\bigl(\|\varphi^n\|^2-\|\varphi^{n-2}\|^2\bigr)
  +\sigma\|\varphi^n\|^2
  +{\cal A}_h(\xi^n,\varphi^n)
  &\le (\mathcal{R}_v^n,\varphi^n).
\end{align}
For the elastic-kinematic coupling, note from \eqref{eq:err-bdf2-u}:
\[
  {\cal A}_h(\xi^n,\varphi^n)
  = {\cal A}_h(\xi^n,D^{(2)}_t\xi^n-\mathcal{R}_{u,h}^n)
  \ge \frac{1}{4\tau}\bigl(\|\xi^n\|^2_{DG}-\|\xi^{n-2}\|^2_{DG}\bigr)
  -{\cal A}_h(\xi^n,\mathcal{R}_{u,h}^n).
\]
Applying Cauchy--Schwarz and Young's inequality:
\begin{align*}
  |(\mathcal{R}_v^n,\varphi^n)|
  &\le \tfrac{\sigma}{2}\|\varphi^n\|^2
       +\tfrac{1}{2\sigma}\|\mathcal{R}_v^n\|^2,\\
  |{\cal A}_h(\xi^n,\mathcal{R}_{u,h}^n)|
  &\le \tfrac14\|\xi^n\|^2_{DG}+C\|\mathcal{R}_{u,h}^n\|^2_{DG}
   \le \tfrac14\|\xi^n\|^2_{DG}+\tfrac{C}{h^2}\|\mathcal{R}_{u}^n\|^2.
\end{align*}
Multiplying \eqref{eq:l2-bdf2-tested} by $4\tau$ and using the residual bounds
derived in the energy norm proof:
\[
  \|\varphi^n\|^2-\|\varphi^{n-2}\|^2+\|\xi^n\|^2_{DG}-\|\xi^{n-2}\|^2_{DG}
  \le C\tau\|\mathcal{R}_v^n\|^2 + \tfrac{C\tau}{h^2}\|\mathcal{R}_u^n\|^2.
\]
Summing from $n=2$ to $N$ and telescoping:
\[
  \|\varphi^N\|^2+\|\varphi^{N-1}\|^2+\|\xi^N\|^2_{DG}+\|\xi^{N-1}\|^2_{DG}
  -\|\varphi^0\|^2-\|\varphi^1\|^2-\|\xi^0\|^2_{DG}-\|\xi^1\|^2_{DG}
  \le C\tau\sum_{n=2}^{N}\|\mathcal{R}_v^n\|^2
   + \tfrac{C\tau}{h^2}\sum_{n=2}^{N}\|\mathcal{R}_u^n\|^2.
\]
Using the residual bounds from the energy norm proof and noting that
$\frac{\tau}{h^2}\cdot h^{2(k+1)}\tau^{-1}=h^{2k}$, we obtain
\begin{align}\label{eq:l2-bdf2-sum}
  \|\varphi^N\|^2+\|\xi^N\|^2_{DG}
  &\le C\bigl(\|\varphi^1\|^2+\|\xi^1\|^2_{DG}\bigr)
  +Ch^{2(k+1)}\bigl(\|u\|^2_{L^2(H^{k+1})}+\|v\|^2_{L^2(H^{k+1})}\bigr)\notag\\
  &\quad +C\tau^4\!\int_0^{t_N}\!\bigl(\|u_{ttt}\|^2+\|v_{ttt}\|^2\bigr)\,\mathrm{d}s
  +C\tau\sum_{n=2}^N\bigl(\|\varphi^n\|^2+\|\xi^n\|^2_{DG}\bigr),
\end{align}
where the last term comes from the nonlinear residual \eqref{eq:nonlin-full}
and the coupling term $\frac{C\tau}{h^2}\|\mathcal{R}_u^n\|^2$ combined with the
energy estimate to give $h^{2k}$ which is bounded by $h^{2(k+1)}$ for $\tau\le Ch^2$.

\bigskip

Adding \eqref{eq:l2-cn-bound}--\eqref{eq:l2-xi1} to \eqref{eq:l2-bdf2-sum},
the initial contributions $\|\xi^1\|^2_{DG}+\|\varphi^1\|^2$ are bounded by
Step~A. For $\tau$ sufficiently small, the discrete Gr\"onwall inequality
absorbs the last sum on the right, yielding
\begin{equation}\label{eq:l2-final}
  \|\xi^N\|^2+\|\varphi^N\|^2
  \le C\bigl(h^{2(k+1)}+\tau^4\bigr).
\end{equation}
Combining with the projection error via the triangle inequality completes
the proof.
\end{proof}

\section{Numerical Experiments}\label{sec:numerics}
We present numerical experiments for the proposed SIPG method applied
to the weakly damped semilinear wave equation~\eqref{eq:model} in
$\Omega\times(0,T]$.  The model problem \eqref{eq:model} is solved
using the CN--BDF2 time integration on discontinuous Galerkin spaces
of polynomial degree~$k$.  Triangular meshes on the unit square
$\Omega=(0,1)^2$ (Sections~\ref{subsec:ex1}--\ref{subsec:ex2}) and on
the extended domain $(-10,10)^2$ (Section~\ref{subsec:sinegordon}) are
considered to assess the method's performance under linear, polynomial,
and trigonometric nonlinearities.

\medskip\noindent\textbf{Implementation details.}
All computations are performed using FEniCS~\cite{fenics2019} with the
SIPG bilinear form $\mathcal{A}_h(\cdot,\cdot)$ and penalty parameter
$\eta=10$.  The nonlinear algebraic system at each time step is
resolved via Picard (fixed-point) iteration.  The fully discrete
scheme, written in operator form, reads:
\begin{equation}\label{eq:fds-num}
\begin{cases}
  \bigl(D^{(2)}_t u_h^n,\,\phi_h\bigr)
  = \bigl(v_h^n,\,\phi_h\bigr),
  & n\ge2,\\[4pt]
  \bigl(D^{(2)}_t v_h^n,\,\psi_h\bigr)
  +\sigma\bigl(v_h^n,\,\psi_h\bigr)
  +\mathcal{A}_h\bigl(u_h^n,\,\psi_h\bigr)
  +\bigl(G(u_h^n,u_h^{n-2}),\,\psi_h\bigr)
  = \bigl(f^n,\,\psi_h\bigr),
  & n\ge2,\\[4pt]
  \bigl(\partial_\tau u_h^1,\,\phi_h\bigr)
  = \bigl(v_h^{1/2},\,\phi_h\bigr),
  & n=1,\\[4pt]
  \bigl(\partial_\tau v_h^1,\,\psi_h\bigr)
  +\sigma\bigl(v_h^{1/2},\,\psi_h\bigr)
  +\mathcal{A}_h\bigl(u_h^{1/2},\,\psi_h\bigr)
  +\bigl(G(u_h^1,u_h^{-1}),\,\psi_h\bigr)
  = \bigl(f^{1/2},\,\psi_h\bigr),
  & n=1,
\end{cases}
\end{equation}
for all $\phi_h,\psi_h\in V_h$.

\subsection{Linear equation with smooth solution}\label{subsec:ex1}

We consider the weakly damped linear wave equation~\eqref{eq:model}
with $g\equiv0$ and the manufactured exact solution
\begin{equation}\label{eq:ex1-exact}
  u(x,y,t) = t^2\sin(\pi x)\sin(\pi y),
\end{equation}
damping coefficient $\sigma=0.05$, and final time $T=0.5$.
The source term $f$ is determined analytically so that
\eqref{eq:ex1-exact} satisfies~\eqref{eq:model} exactly.
Since $g=0$, the scheme reduces to a linear system at each time level.

\subsubsection{Spatial convergence}

We verify the spatial accuracy of the SIPG discretization using DG
polynomial degree $k=1$ on a sequence of uniformly refined triangular
meshes with $M=8,16,32,64,128$ subdivisions per coordinate direction
($h=\sqrt{2}/M$).  The time step is chosen as
$\tau=h^{k+1}/2=h^2/2$, ensuring that the temporal truncation error
remains negligible relative to the spatial discretization error.

\begin{table}[htbp]
\centering
\caption{Spatial convergence for DG($k{=}1$),
$u=t^2\sin(\pi x)\sin(\pi y)$, $\sigma=0.05$, $g=0$, $\tau=h^2/2$.}
\label{tab:ex1}
\begin{tabular}{r c c c c c c}
\toprule
$M$ & $h$ & $\tau$ & $\|e_h\|$ & rate
    & $\|e_h\|_{DG}$ & rate \\
\midrule
  8  & 1.768e-01 & 1.562e-02 & 1.630e-03 & ---  & 7.189e-02 & --- \\
 16  & 8.839e-02 & 3.906e-03 & 4.241e-04 & 1.94 & 3.342e-02 & 1.11 \\
 32  & 4.419e-02 & 9.766e-04 & 1.085e-04 & 1.97 & 1.616e-02 & 1.05 \\
 64  & 2.210e-02 & 2.441e-04 & 2.747e-05 & 1.98 & 7.954e-03 & 1.02 \\
128  & 1.105e-02 & 6.104e-05 & 6.914e-06 & 1.99 & 3.948e-03 & 1.01 \\
\bottomrule
\end{tabular}
\end{table}

Table~\ref{tab:ex1} confirms that the $L^2$-error converges at the
optimal rate $\mathcal{O}(h^{k+1})=\mathcal{O}(h^2)$ and the discrete
energy-norm error at rate $\mathcal{O}(h^k)=\mathcal{O}(h)$, in
excellent agreement with the theoretical predictions of
Theorems~\ref{thm:error} and~\ref{thm:l2-error}.  The asymptotic
regime is clearly attained for $M\geq 16$.

\subsubsection{Temporal convergence}

To isolate the temporal accuracy of the CN--BDF2 integrator, we fix a
sufficiently fine spatial mesh ($M=128$, $h\approx1.1\times10^{-2}$)
so that the spatial discretization error is negligible, and
successively halve the time step~$\tau$.

\begin{table}[htbp]
\centering
\caption{Temporal convergence of the CN--BDF2 scheme
($M=128$, $k=1$, $T=0.5$).}
\label{tab:ex1-temporal}
\begin{tabular}{c c c}
\toprule
$\tau$ & $\|e^n_h\|$ & rate \\
\midrule
$1.0000\times 10^{-1}$ & 6.7231e-03 & ---  \\
$5.0000\times 10^{-2}$ & 1.8659e-03 & 1.85 \\
$2.5000\times 10^{-2}$ & 4.8922e-04 & 1.93 \\
$1.2500\times 10^{-2}$ & 1.2918e-04 & 1.92 \\
$6.2500\times 10^{-3}$ & 3.8088e-05 & 1.76 \\
\bottomrule
\end{tabular}
\end{table}

Table~\ref{tab:ex1-temporal} confirms second-order temporal accuracy.
The observed rates approach the theoretical value~$2$ for moderate time
steps and exhibit a mild plateau on the finest level, where the
temporal error becomes comparable to the residual spatial
discretization error---a standard artifact in mixed space--time
convergence studies.

\subsubsection{Solution profiles and Lyapunov functional}

\begin{figure}[htbp]
\centering
\begin{minipage}[t]{0.48\textwidth}
  \centering
  \includegraphics[width=\textwidth]{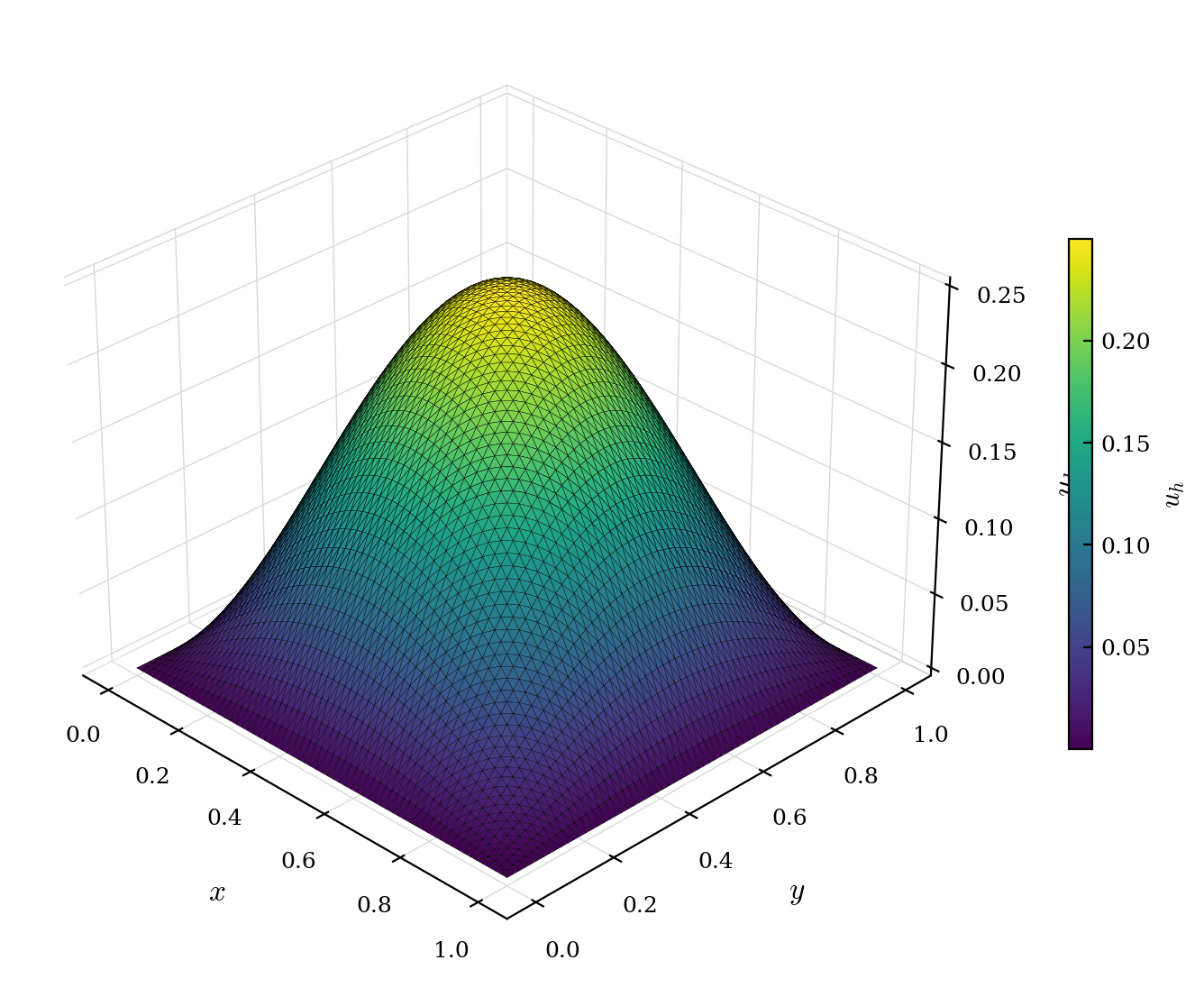}
  \caption{Linear wave equation: Surface plot of $u_h$ at $T=0.5$ ($M=128$, $k=1$).}
  \label{fig:ex1-surface}
\end{minipage}\hfill
\begin{minipage}[t]{0.48\textwidth}
  \centering
  \includegraphics[width=\textwidth]{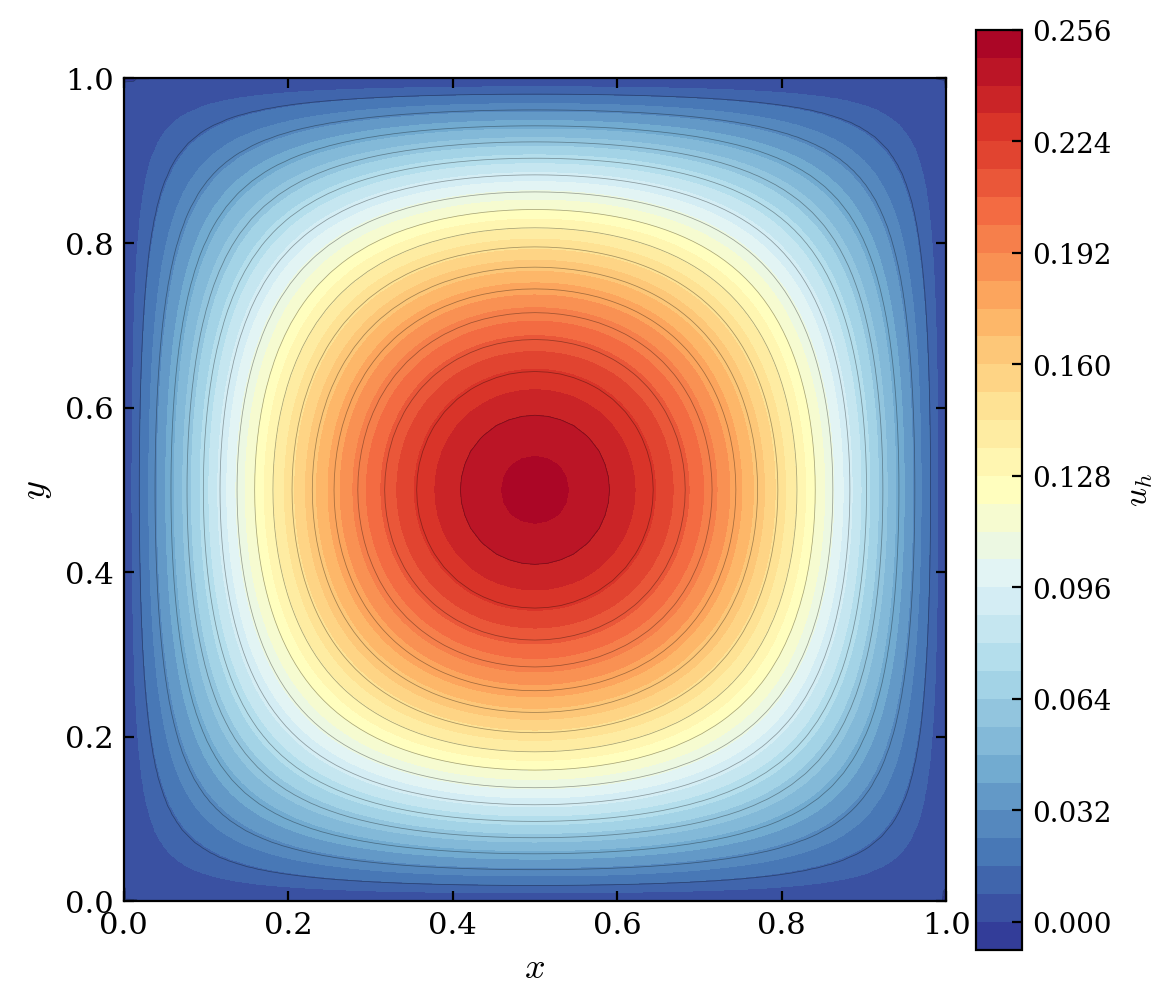}
  \caption{Linear wave equation: Contour plot of $u_h$ at $T=0.5$ ($M=128$, $k=1$).}
  \label{fig:ex1-contour}
\end{minipage}
\end{figure}

Figures~\ref{fig:ex1-surface} and~\ref{fig:ex1-contour} display
surface and filled-contour visualizations of the discrete solution at
$T=0.5$ on a mesh with $M=128$.  The smooth sinusoidal profile and
four-fold symmetry of the exact solution are faithfully reproduced by
the SIPG approximation.

\begin{figure}[htbp]
\centering
\includegraphics[width=0.65\textwidth]{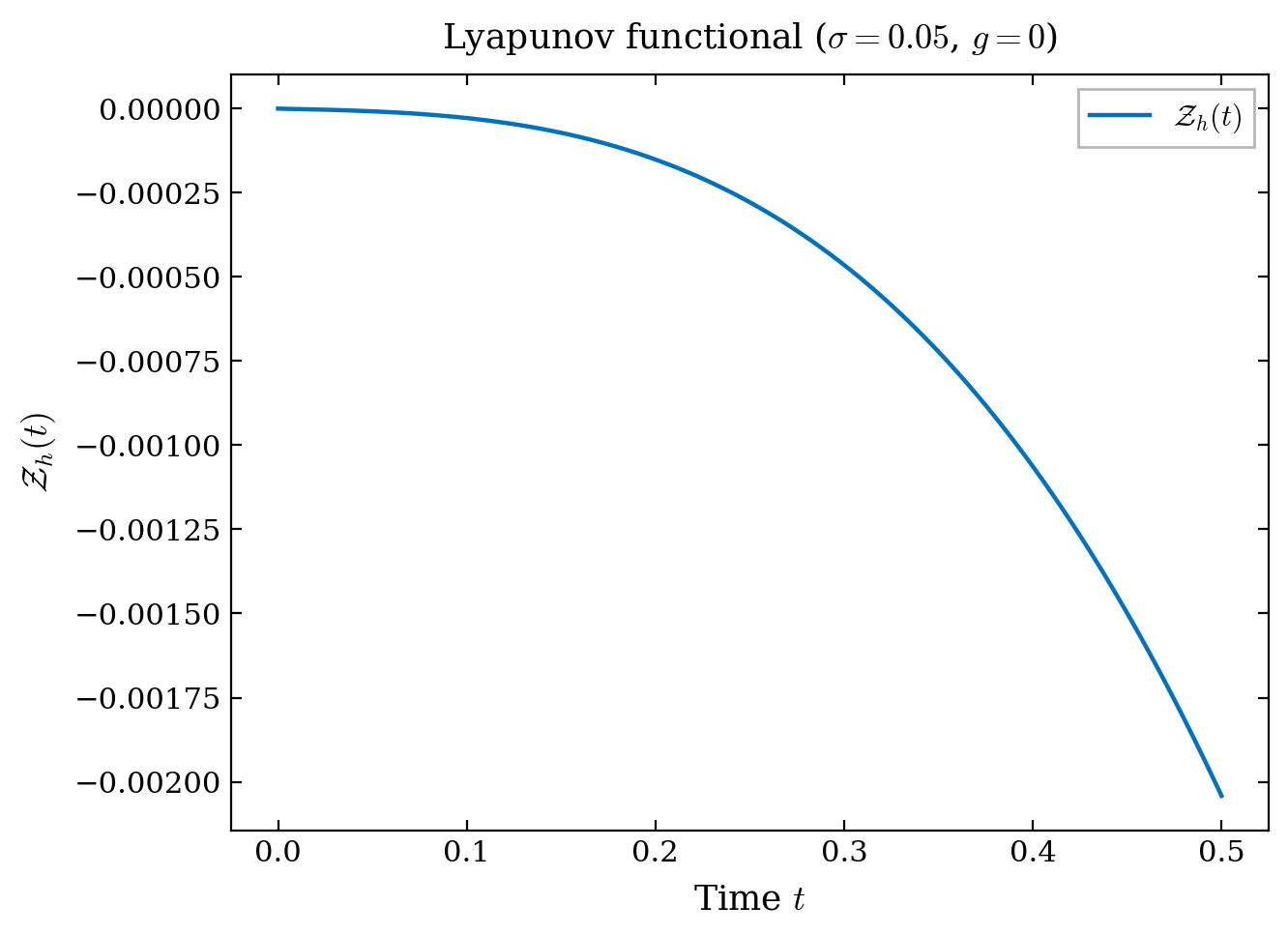}
\caption{Linear wave equation: Evolution of the discrete Lyapunov functional
$\mathcal{Z}_h(t)$ over $[0,0.5]$ ($\sigma=0.05$, $g=0$).}
\label{fig:ex1-lyapunov}
\end{figure}

Figure~\ref{fig:ex1-lyapunov} shows the time history of the discrete
Lyapunov functional $\mathcal{Z}_h(t)$.  In the presence of the
external source~$f$, the functional evolves smoothly and its
trajectory is consistent with the discrete energy identity established
in Lemma~\ref{lem:energy}.

\subsection{Cubic nonlinearity}\label{subsec:ex2}

We now assess the robustness of the CN--BDF2 SIPG scheme in the
nonlinear regime by considering the cubic reaction term $g(u)=u^3$.
The manufactured exact solution is
\begin{equation}\label{eq:ex2-exact}
  u(x,y,t) = e^{t}\,x\,y\,(1-x)(1-y),
\end{equation}
with damping coefficient $\sigma=1$ and final time $T=0.5$.  The
primitive is $F(s)=s^4/4$, and the chord-slope operator reads
$G(a,b)=(a^2+b^2)(a+b)/4$ for $a\neq b$.

We employ DG polynomial degree $k=1$ on the same mesh sequence as in
Section~\ref{subsec:ex1}.  To accommodate the additional stiffness
introduced by the cubic nonlinearity, the time step is reduced to
$\tau=h^2/3$.

\begin{table}[htbp]
\centering
\caption{Spatial convergence for DG($k{=}1$),
$u=e^t xy(1{-}x)(1{-}y)$, $\sigma=1$, $g(u)=u^3$, $\tau=h^2/3$.}
\label{tab:ex2}
\begin{tabular}{r c c c c c c}
\toprule
$M$ & $h$ & $\tau$ & $\|e_h\|$ & rate
    & $\|e_h\|_{DG}$ & rate \\
\midrule
  8 & $1.768\times10^{-1}$ & $1.042\times10^{-2}$
    & $1.236\times10^{-3}$ & ---
    & $3.795\times10^{-2}$ & --- \\
 16 & $8.839\times10^{-2}$ & $2.604\times10^{-3}$
    & $3.450\times10^{-4}$ & $1.84$
    & $1.807\times10^{-2}$ & $1.07$ \\
 32 & $4.419\times10^{-2}$ & $6.510\times10^{-4}$
    & $9.062\times10^{-5}$ & $1.93$
    & $7.707\times10^{-3}$ & $1.23$ \\
 64 & $2.210\times10^{-2}$ & $1.628\times10^{-4}$
    & $2.328\times10^{-5}$ & $1.96$
    & $3.416\times10^{-3}$ & $1.17$ \\
128 & $1.105\times10^{-2}$ & $4.069\times10^{-5}$
    & $5.901\times10^{-6}$ & $1.98$
    & $1.428\times10^{-3}$ & $1.26$ \\
\bottomrule
\end{tabular}
\end{table}

Table~\ref{tab:ex2} demonstrates that the optimal convergence rates
$\mathcal{O}(h^{k+1})$ in $L^2$ and $\mathcal{O}(h^k)$ in the
discrete energy norm are preserved under the cubic nonlinearity.  The
$L^2$ rates increase monotonically toward~$2$ with mesh refinement,
while the energy-norm rates cluster near~$1$.  The mild oscillation in
the energy rates on coarse meshes ($1.07$--$1.26$) is characteristic
of nonlinear problems and does not indicate any loss of asymptotic
optimality.

\subsubsection{Solution profiles and Lyapunov functional}

\begin{figure}[htbp]
\centering
\begin{minipage}[t]{0.48\textwidth}
  \centering
  \includegraphics[width=\textwidth]{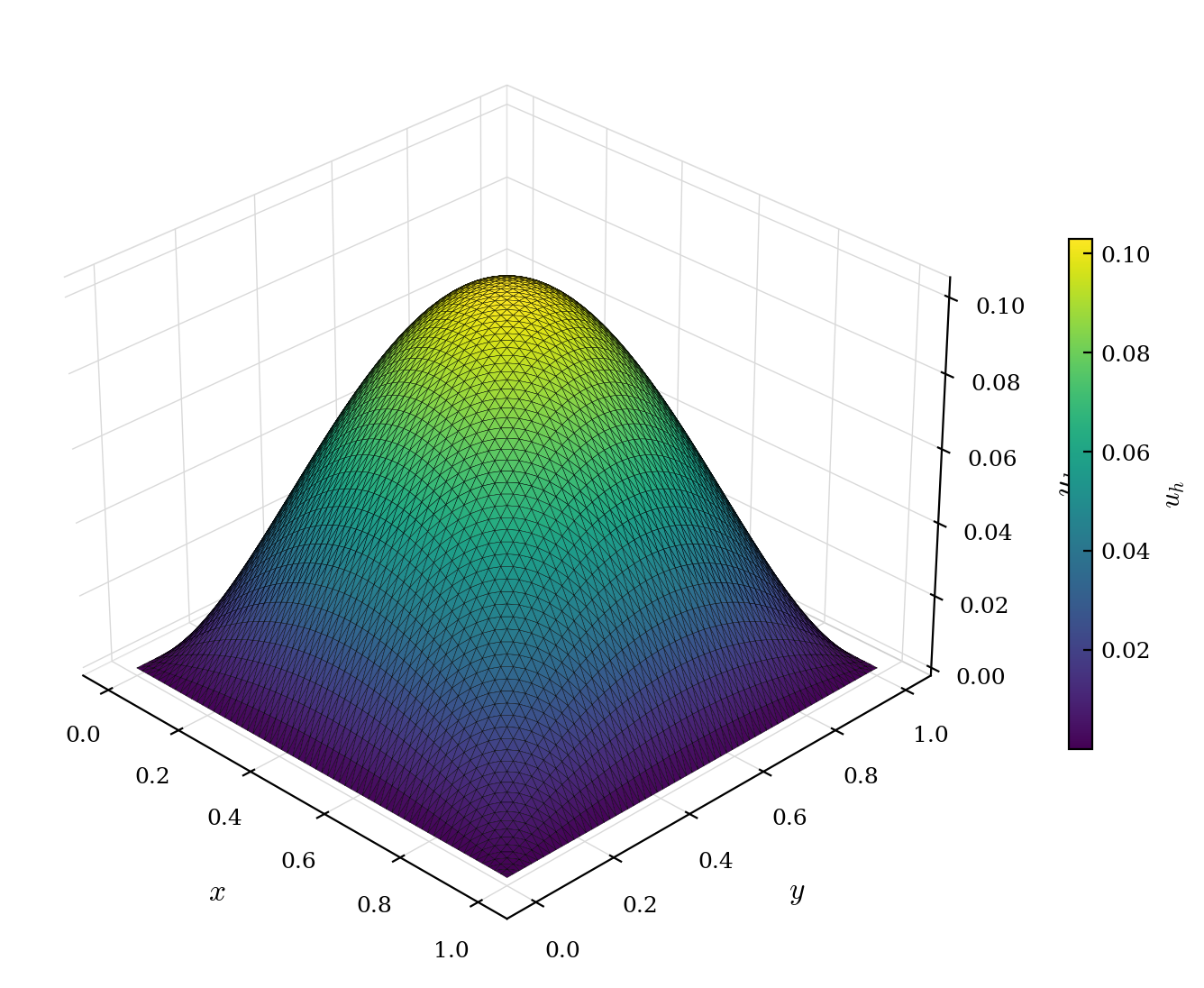}
  \caption{Cubic nonlinearity: Surface plot of $u_h$ at $T=0.5$ ($M=128$, $k=1$,
  $g(u)=u^3$).}
  \label{fig:ex2-surface}
\end{minipage}\hfill
\begin{minipage}[t]{0.48\textwidth}
  \centering
  \includegraphics[width=\textwidth]{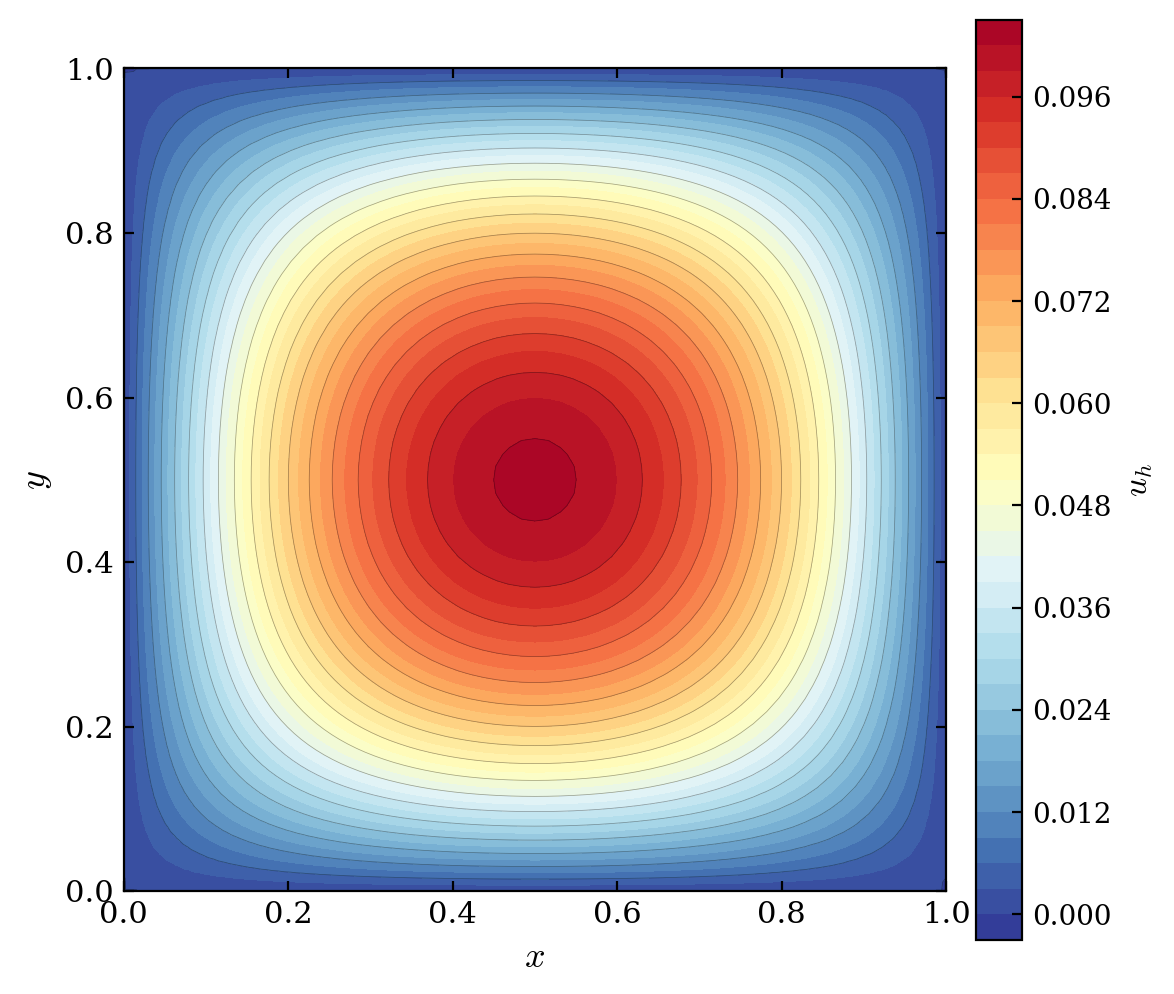}
  \caption{Cubic nonlinearity: Contour plot of $u_h$ at $T=0.5$ ($M=128$, $k=1$,
  $g(u)=u^3$).}
  \label{fig:ex2-contour}
\end{minipage}
\end{figure}

Figures~\ref{fig:ex2-surface} and~\ref{fig:ex2-contour} display the
numerical solution on the finest mesh.  The discrete solution is
smooth and exhibits the expected diagonal symmetry
$u(x,y,t)=u(y,x,t)$, confirming that the SIPG scheme preserves the
qualitative structure of the exact solution in the nonlinear regime.

\begin{figure}[htbp]
\centering
\includegraphics[width=0.65\textwidth]{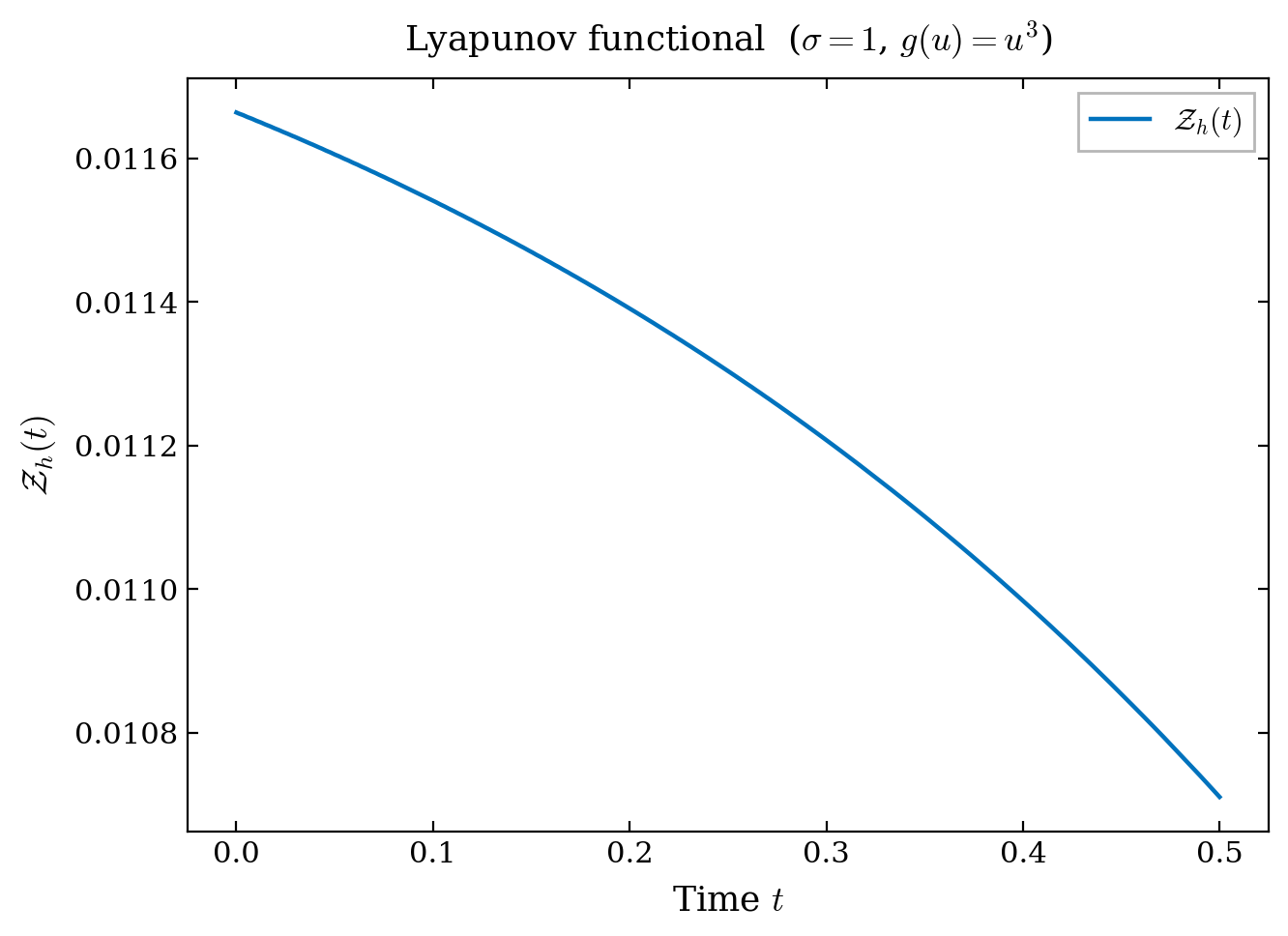}
\caption{Cubic nonlinearity: Evolution of the discrete Lyapunov functional
$\mathcal{Z}_h(t)$ over $[0,0.5]$ ($\sigma=1$, $g(u)=u^3$).}
\label{fig:ex2-lyapunov}
\end{figure}

Figure~\ref{fig:ex2-lyapunov} plots the evolution of the discrete
Lyapunov functional $\mathcal{Z}_h(t)$ for the cubic nonlinearity.  The stronger
damping ($\sigma=1$) combined with the dissipative contribution of the
cubic potential yields a smooth, monotonically evolving energy
trajectory, in agreement with the identity of
Lemma~\ref{lem:energy}.

\subsection{Sine-Gordon equation with energy decay}%
\label{subsec:sinegordon}
As a final test, we consider the damped sine-Gordon equation on
$\Omega=(-10,10)^2$ with homogeneous Neumann boundary conditions:
\begin{equation}\label{eq:sg}
  u_{tt} + \sigma\,u_t - \Delta u + \sin(u) = 0
  \quad\text{in }\Omega\times(0,T],\qquad
  \frac{\partial u}{\partial\mathbf{n}} = 0
  \ \text{on }\partial\Omega,
\end{equation}
with the two-kink initial datum
\begin{equation}\label{eq:sg-ic}
  u(x,y,0) = 4\bigl(\arctan(e^x) + \arctan(e^y)\bigr),
  \qquad u_t(x,y,0) = 0.
\end{equation}
Since $g(s)=\sin s$ is smooth,
the energy estimates of Lemma~\ref{lem:energy} hold without modification.
We set $M=40$ ($h=0.5$), $k=1$, $\tau=0.05$, $T=10$,
and compare the undamped ($\sigma=0$) and damped ($\sigma=1$) regimes.

Figures~\ref{fig:sg-undamped}--\ref{fig:sg-damped} display
surface and contour plots at $t=2,6,10$.
When $\sigma=0$, the kink fronts interact elastically near $t\approx 6$
and separate with unchanged amplitude,
while the solution profiles at $t=2$ and $t=10$ are symmetric —
consistent with energy conservation.
When $\sigma=1$, the kink amplitudes visibly decrease with time
and the fronts progressively broaden,
reflecting the monotone energy dissipation (Lemma \ref{lem:energy}).

This observation is confirmed quantitatively by the discrete
Lyapunov functional plotted in Figure~\ref{fig:sg-lyapunov}:
$\mathcal{Z}_h$ remains constant (up to solver tolerance) for $\sigma=0$
and decays monotonically for $\sigma=1$,
verifying that the CN--BDF2 SIPG scheme preserves the correct
energy structure at the fully discrete level.

\begin{figure}[htbp]
\centering
\begin{minipage}[t]{0.28\textwidth}\centering
  \includegraphics[width=\textwidth]{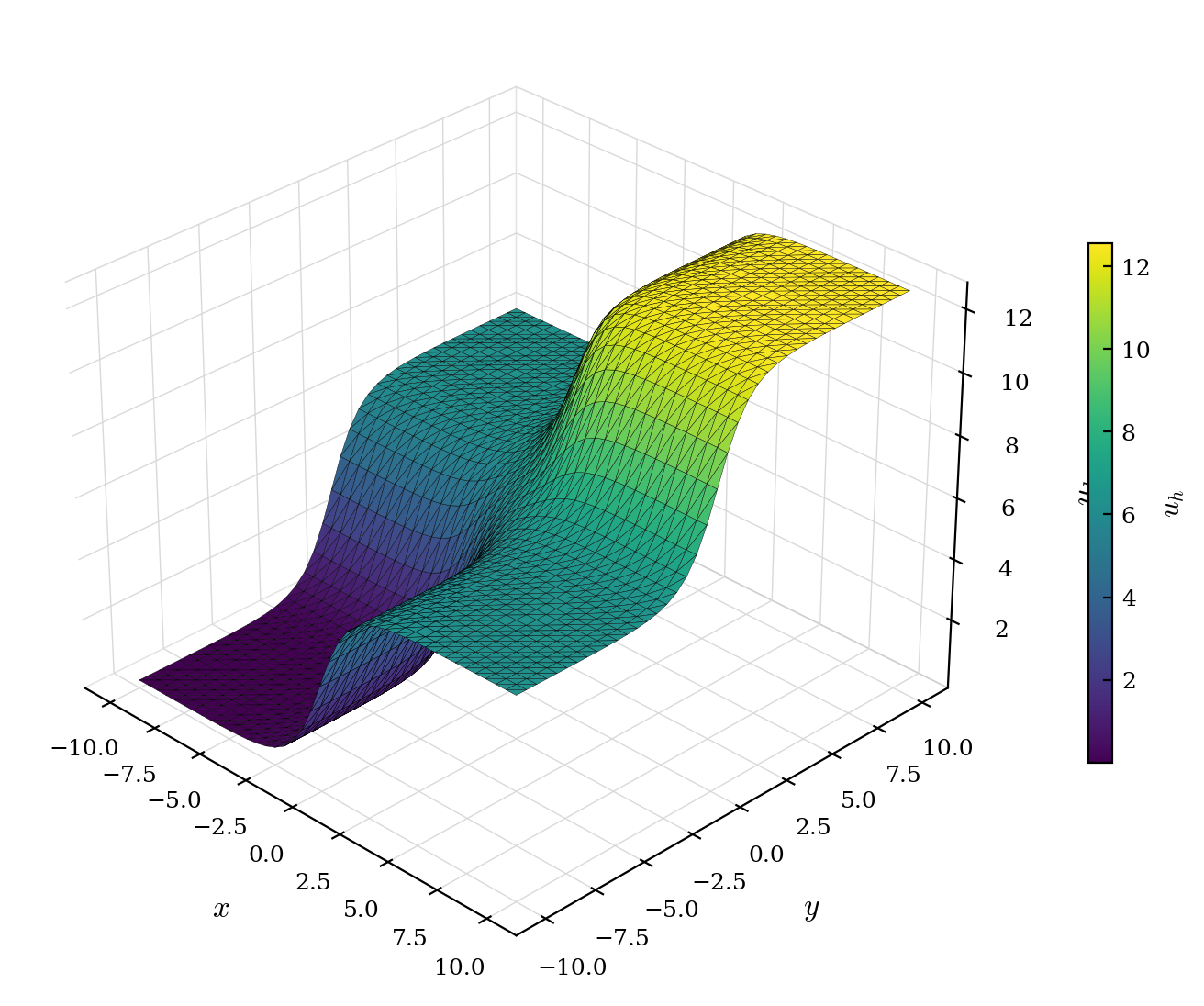}
  \caption*{(a) Surface, $t=2$}
\end{minipage}\hfill
\begin{minipage}[t]{0.28\textwidth}\centering
  \includegraphics[width=\textwidth]{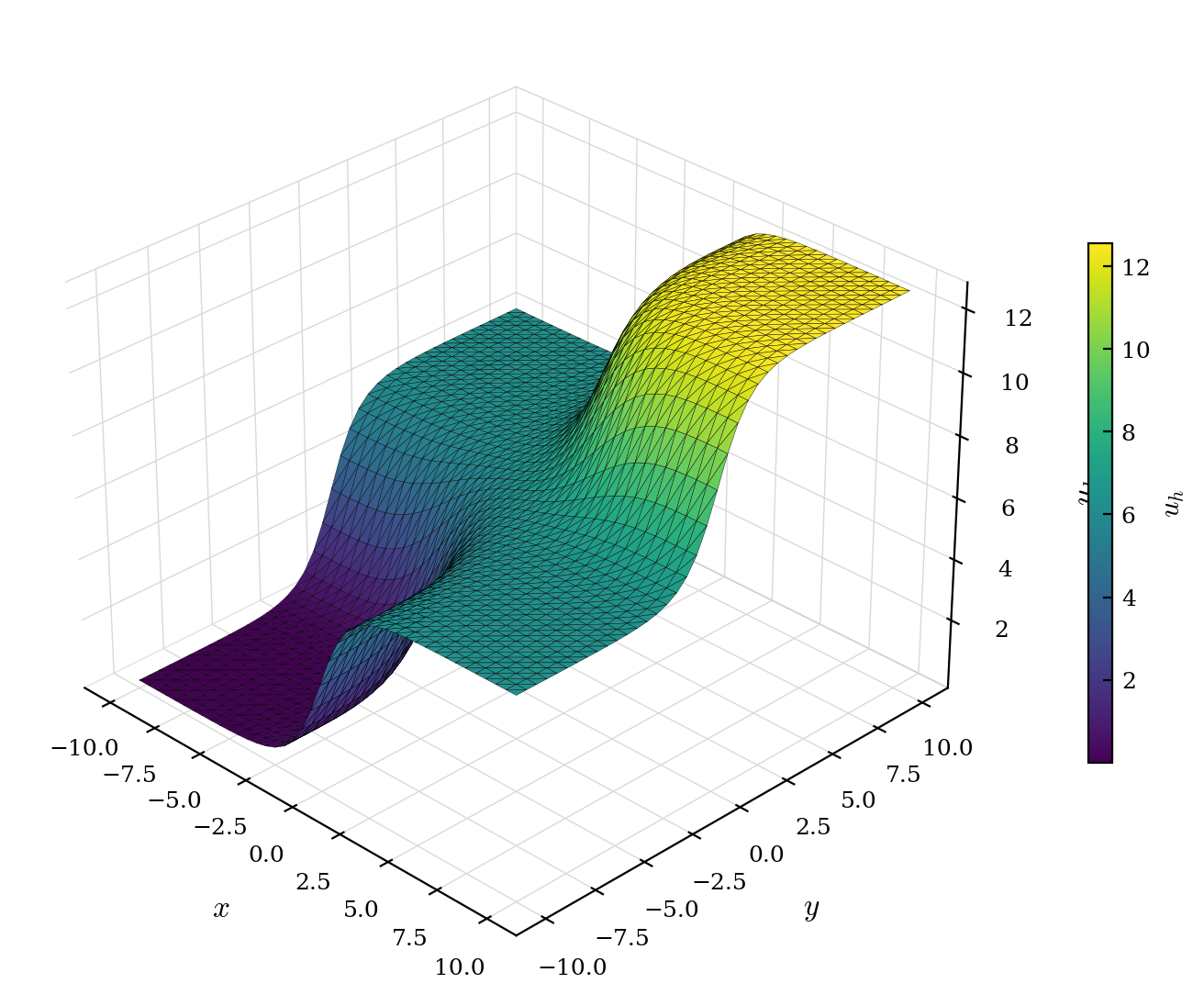}
  \caption*{(b) Surface, $t=6$}
\end{minipage}\hfill
\begin{minipage}[t]{0.28\textwidth}\centering
  \includegraphics[width=\textwidth]{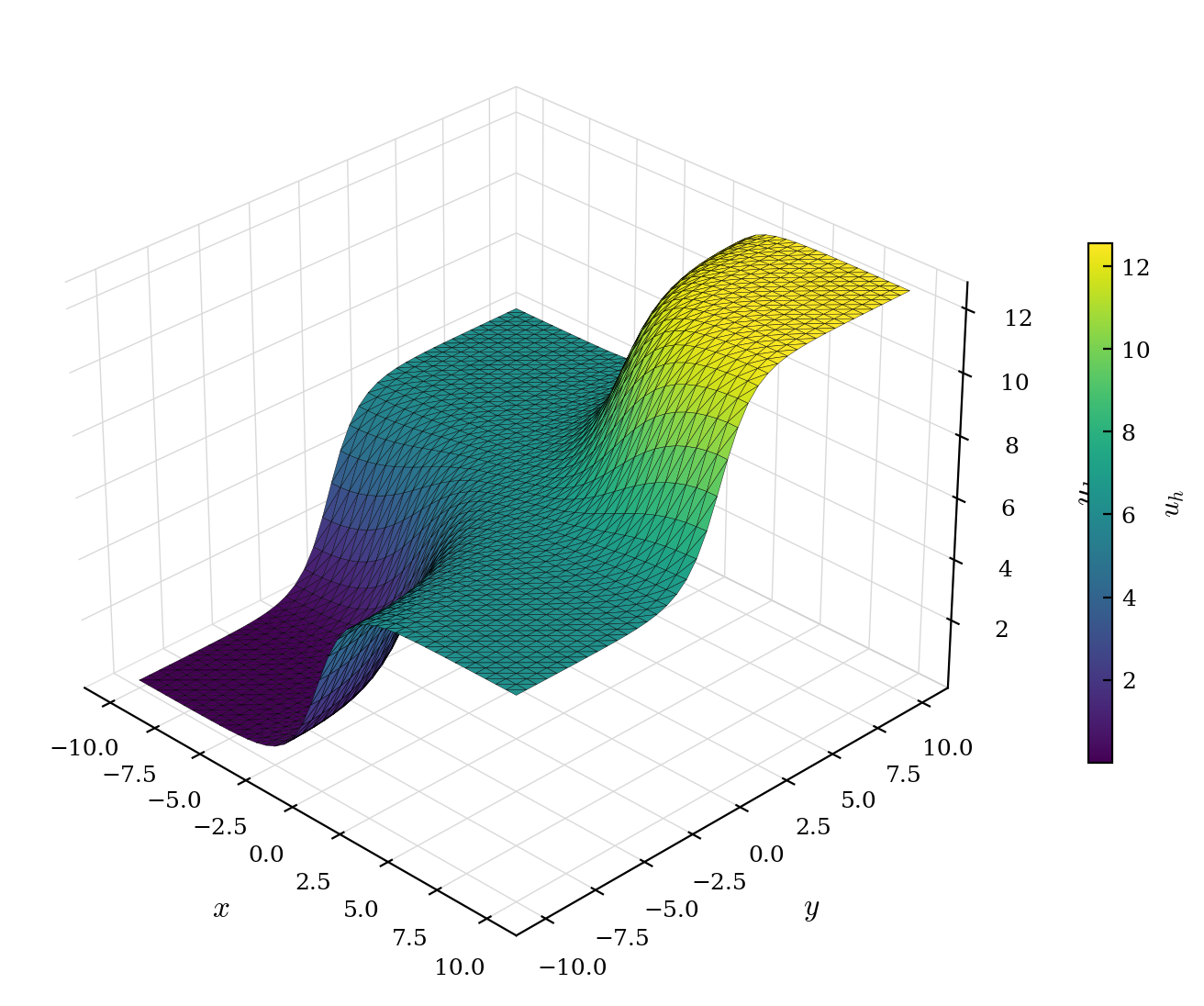}
  \caption*{(c) Surface, $t=10$}
\end{minipage}
\vspace{4pt}
\begin{minipage}[t]{0.26\textwidth}\centering
  \includegraphics[width=\textwidth]{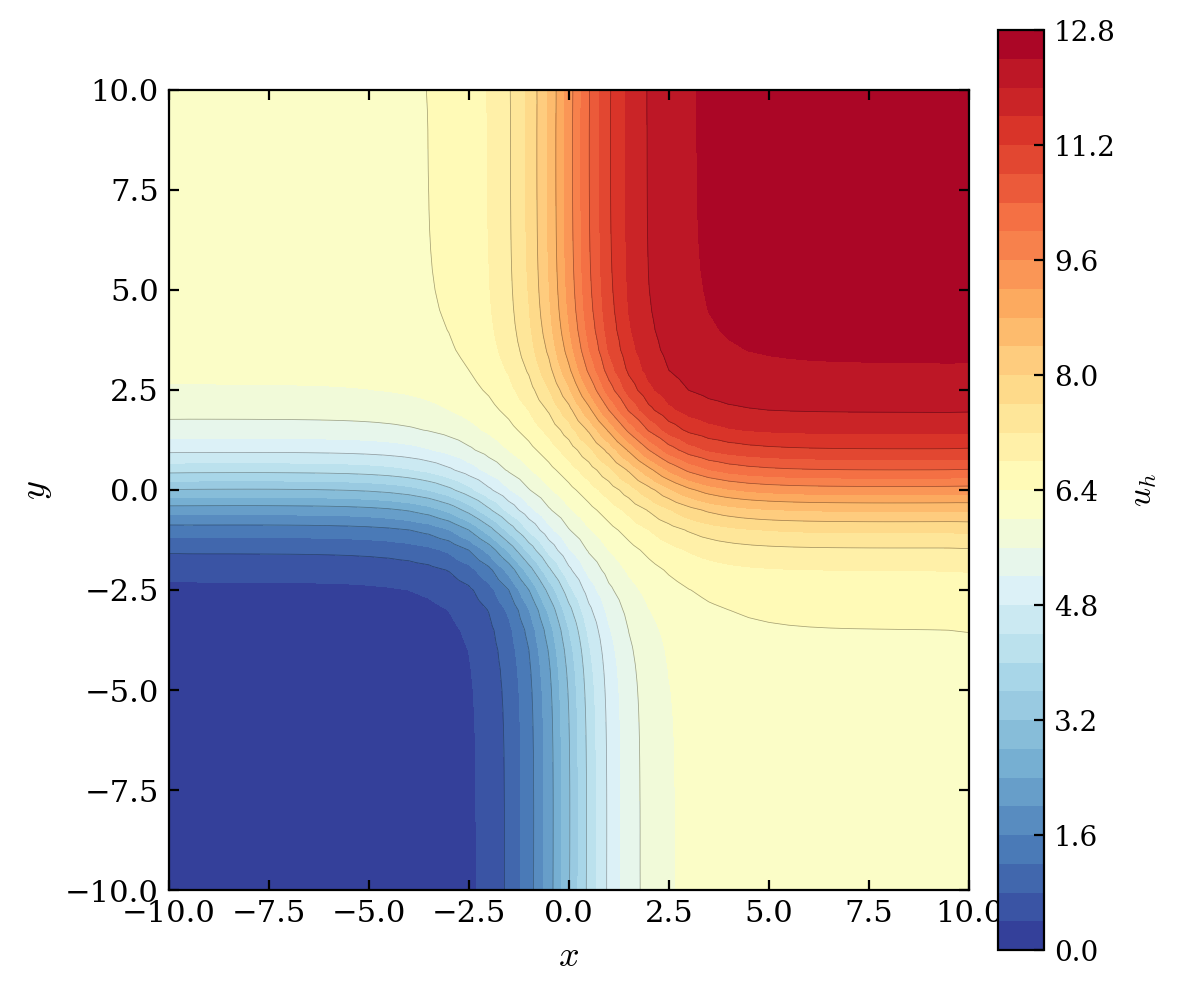}
  \caption*{(d) Contour, $t=2$}
\end{minipage}\hfill
\begin{minipage}[t]{0.26\textwidth}\centering
  \includegraphics[width=\textwidth]{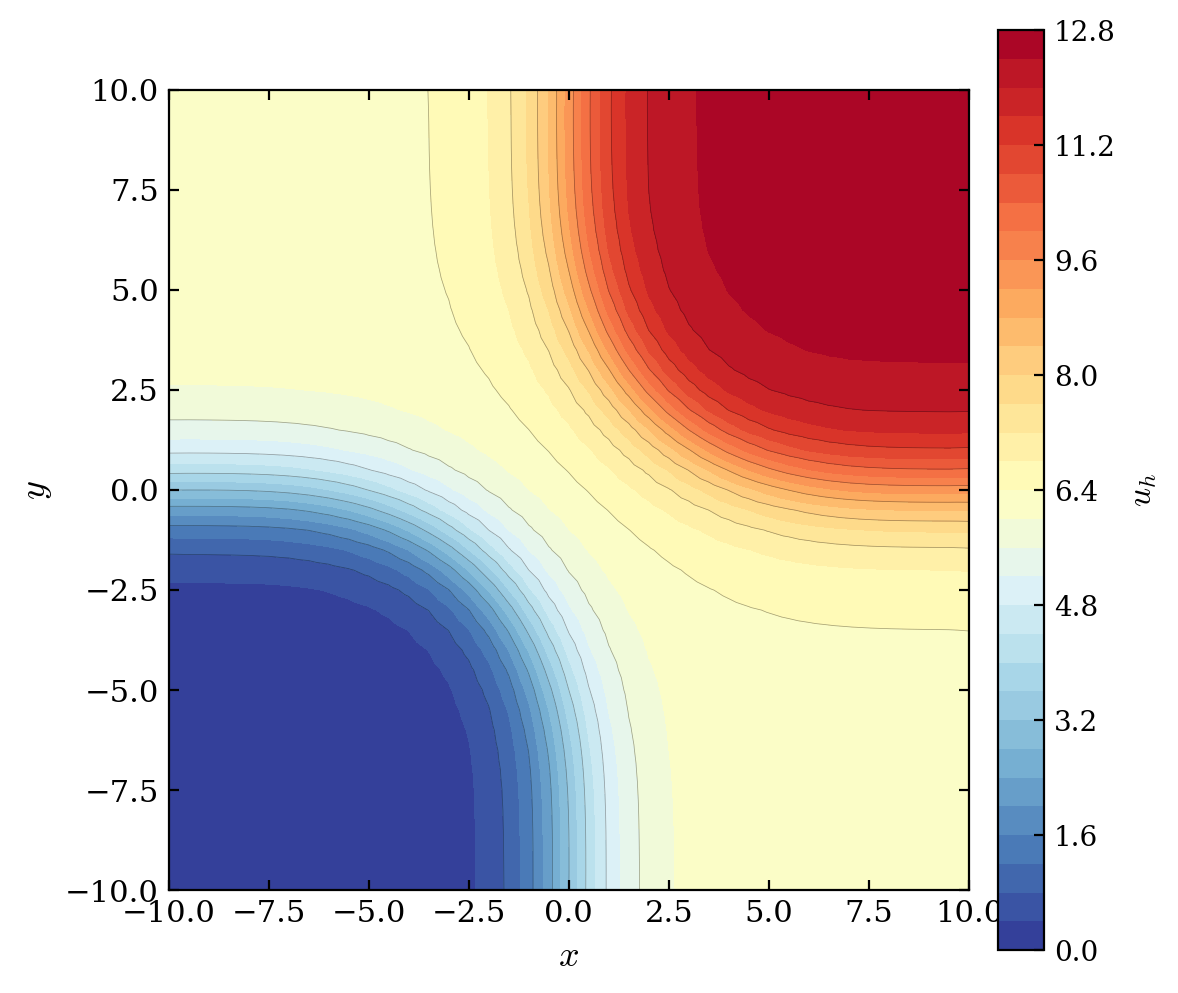}
  \caption*{(e) Contour, $t=6$}
\end{minipage}\hfill
\begin{minipage}[t]{0.26\textwidth}\centering
  \includegraphics[width=\textwidth]{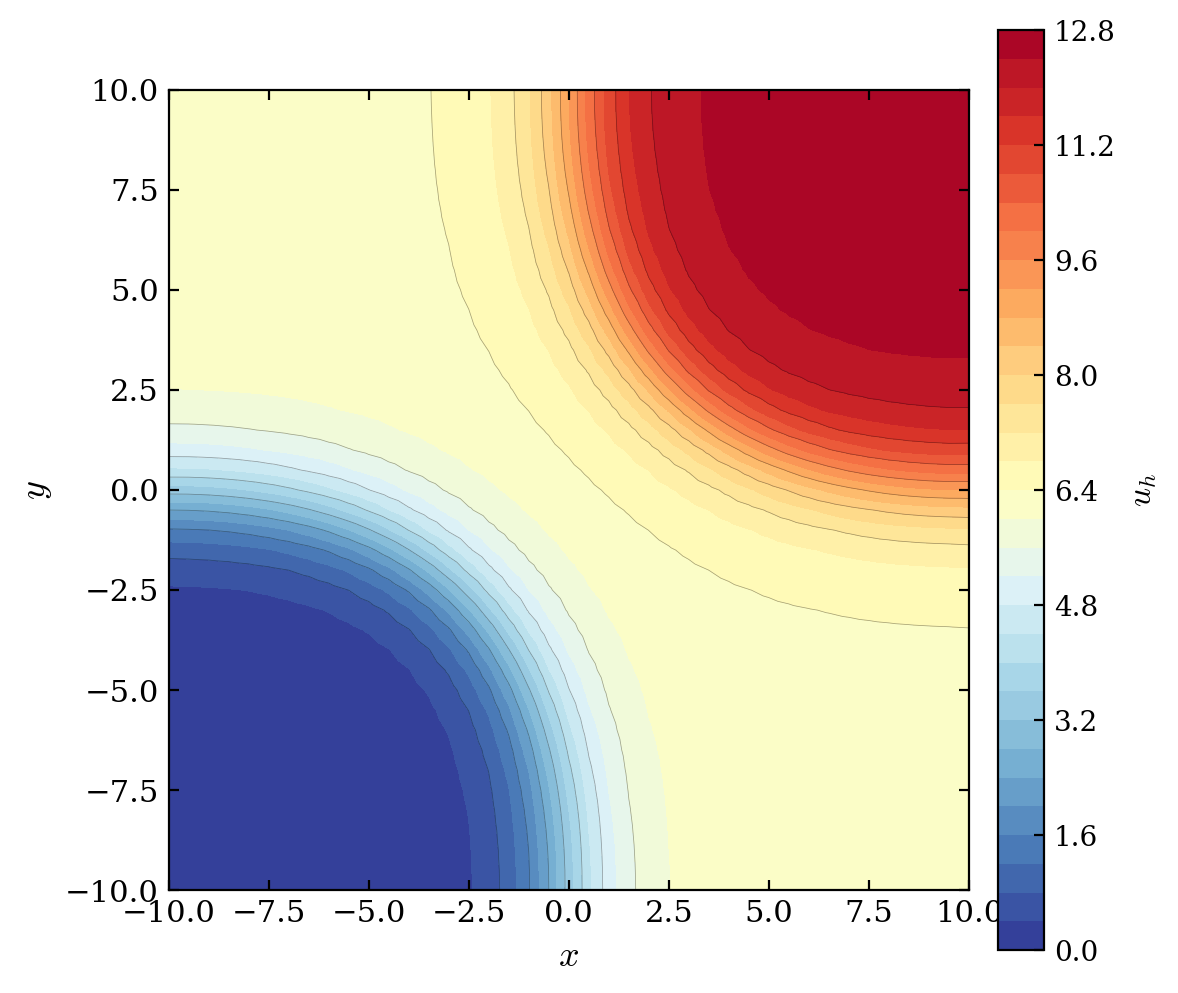}
  \caption*{(f) Contour, $t=10$}
\end{minipage}
\caption{Undamped sine-Gordon ($\sigma=0$): elastic kink--kink interaction with conserved amplitude.}
\label{fig:sg-undamped}
\end{figure}

\begin{figure}[htbp]
\centering
\begin{minipage}[t]{0.28\textwidth}\centering
  \includegraphics[width=\textwidth]{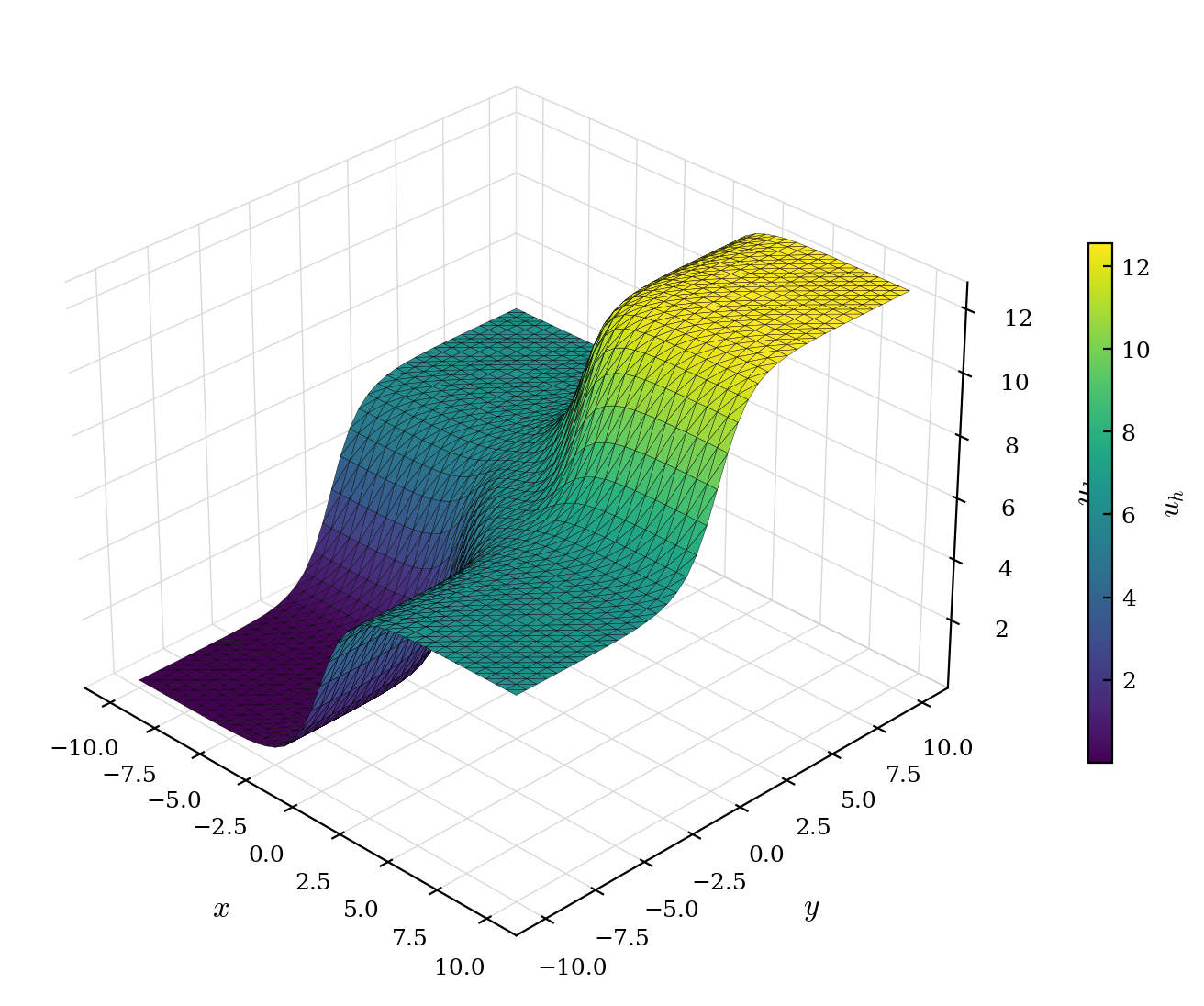}
  \caption*{(a) Surface, $t=2$}
\end{minipage}\hfill
\begin{minipage}[t]{0.28\textwidth}\centering
  \includegraphics[width=\textwidth]{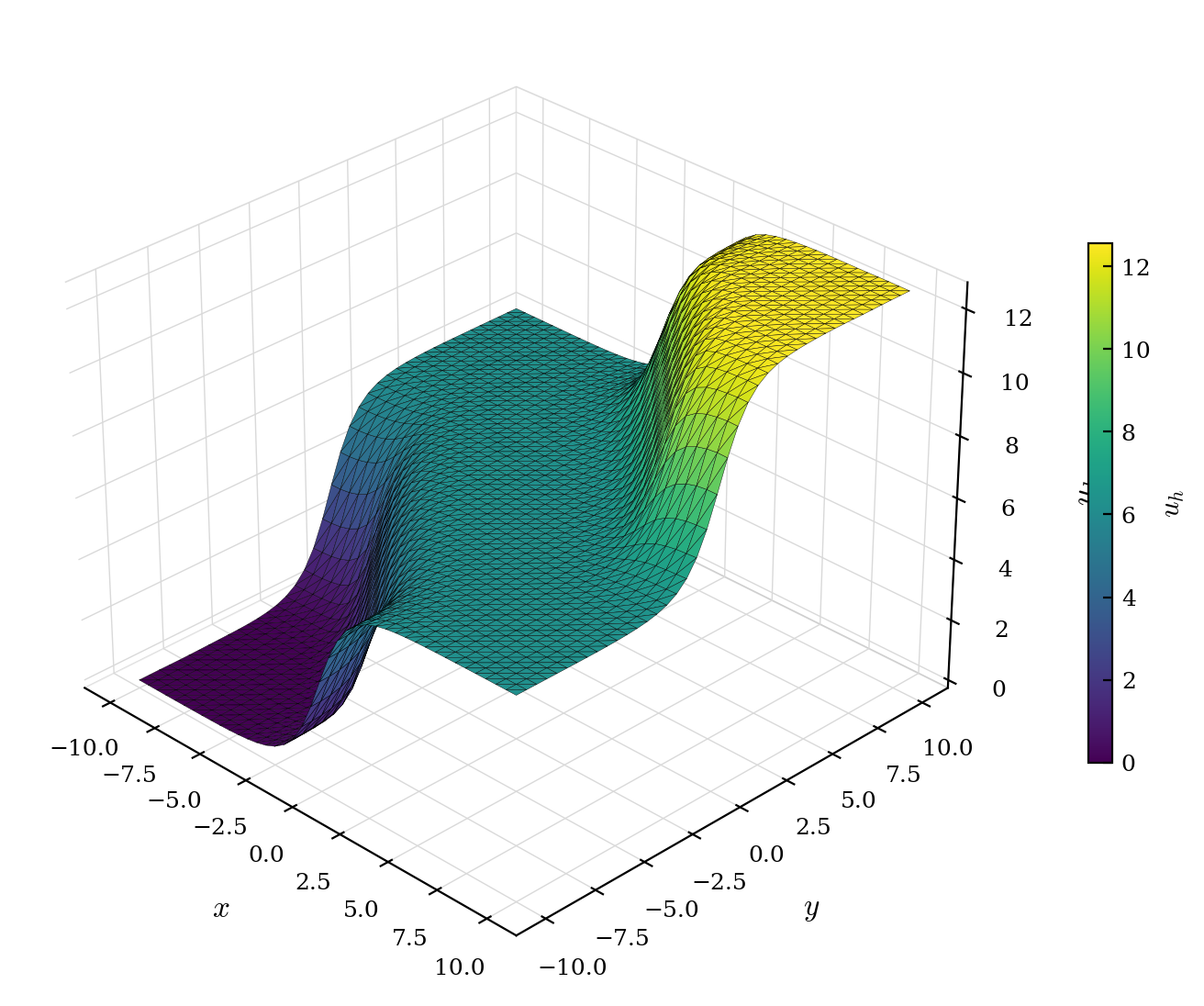}
  \caption*{(b) Surface, $t=6$}
\end{minipage}\hfill
\begin{minipage}[t]{0.28\textwidth}\centering
  \includegraphics[width=\textwidth]{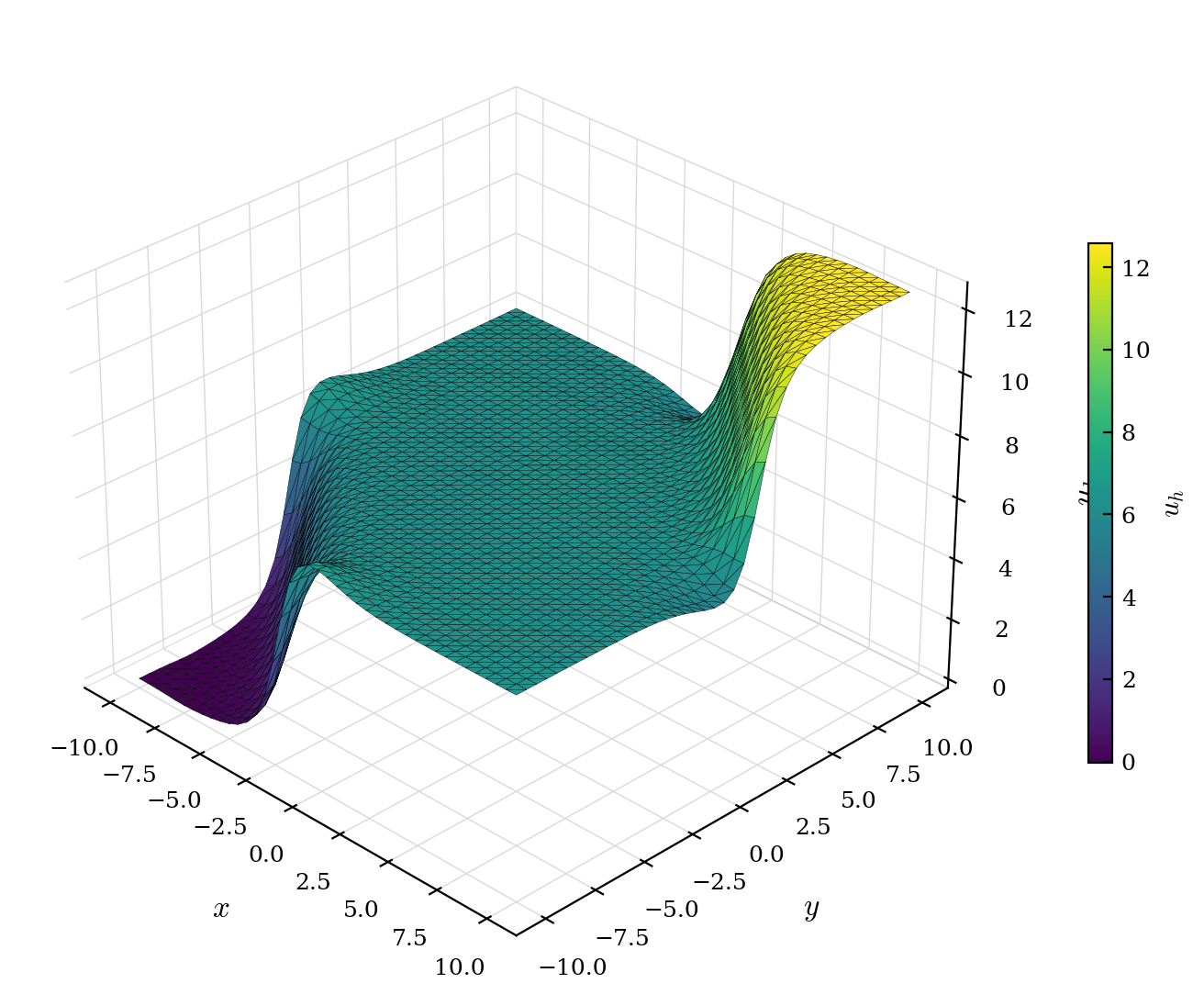}
  \caption*{(c) Surface, $t=10$}
\end{minipage}
\vspace{4pt}
\begin{minipage}[t]{0.26\textwidth}\centering
  \includegraphics[width=\textwidth]{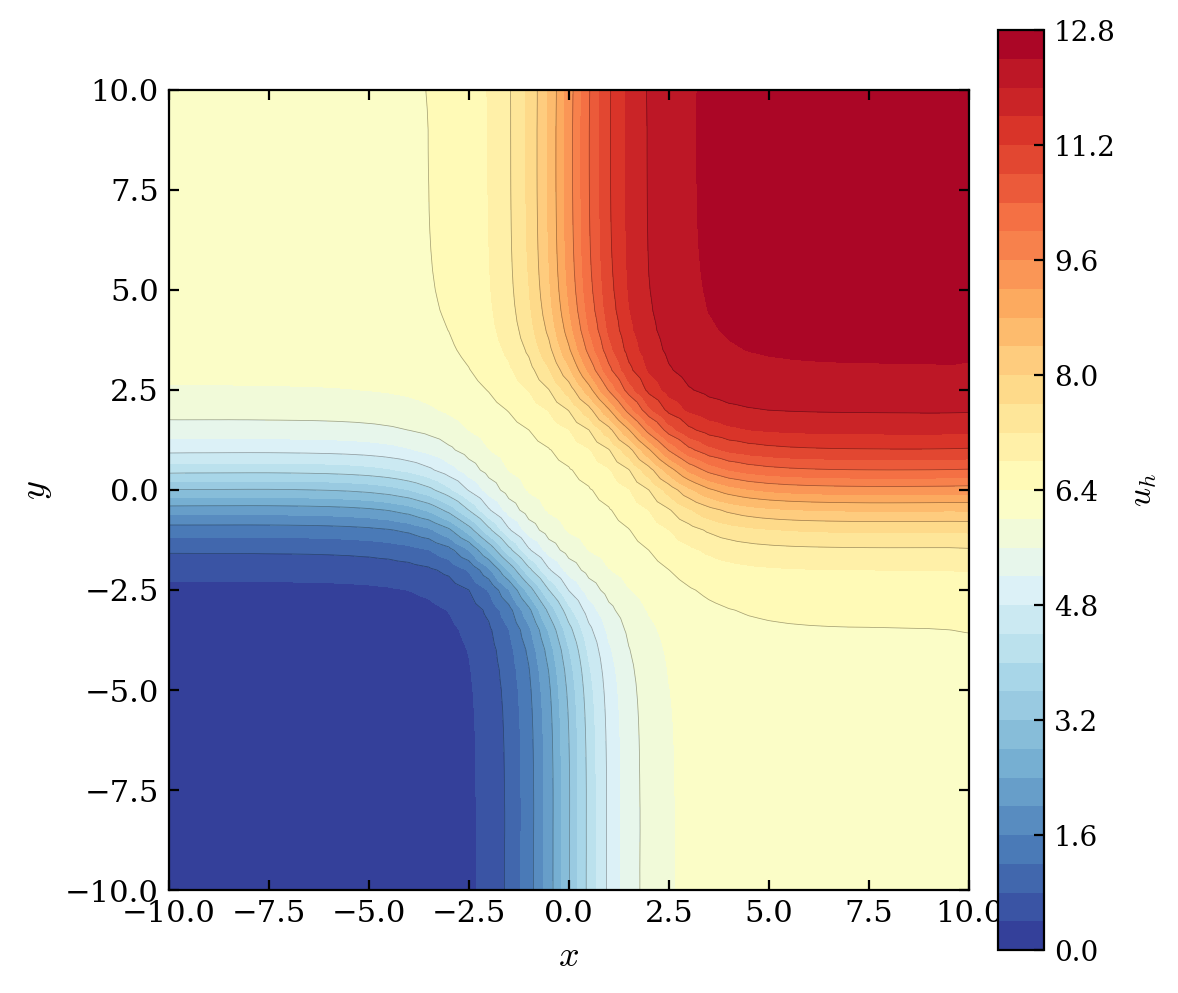}
  \caption*{(d) Contour, $t=2$}
\end{minipage}\hfill
\begin{minipage}[t]{0.26\textwidth}\centering
  \includegraphics[width=\textwidth]{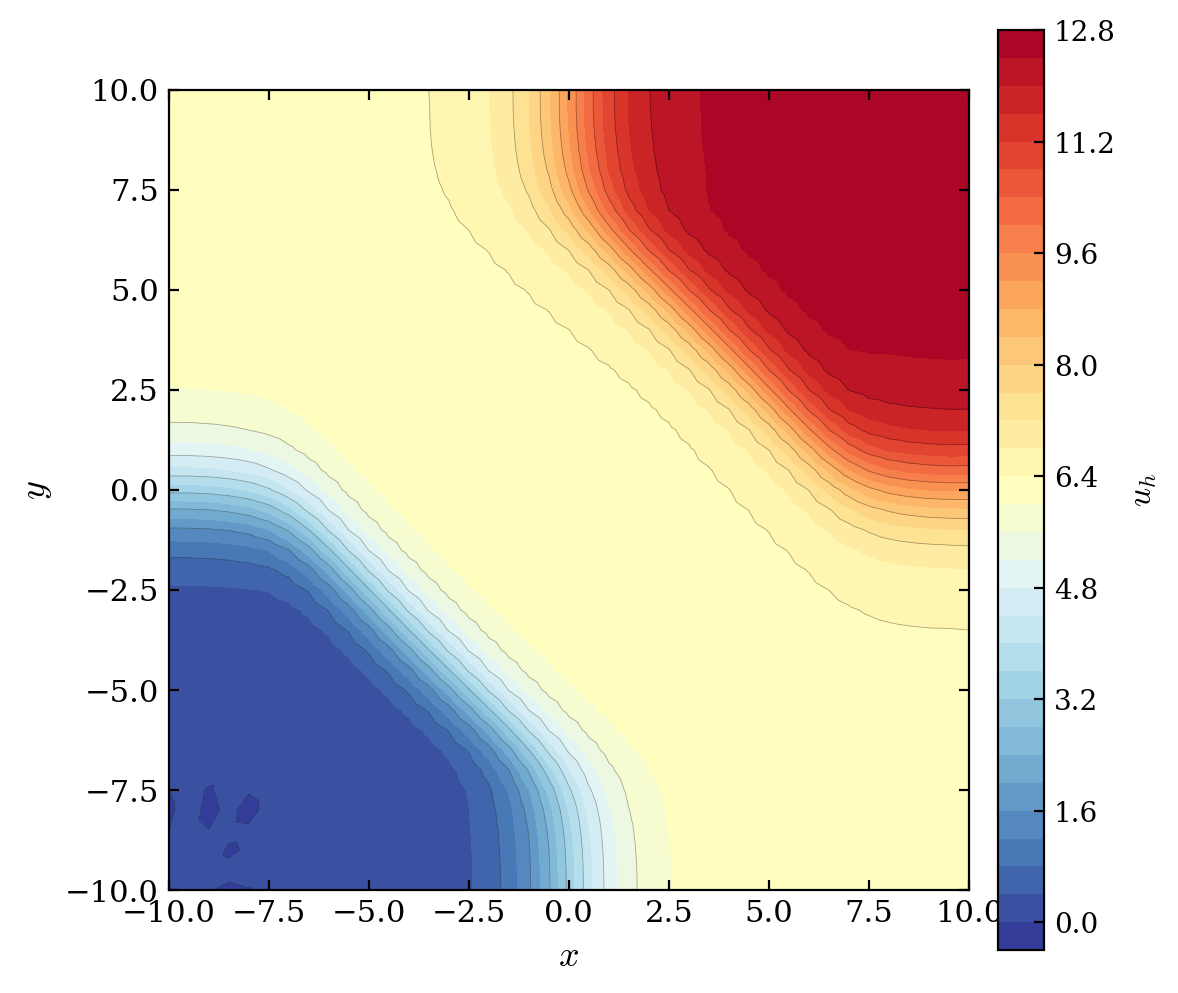}
  \caption*{(e) Contour, $t=6$}
\end{minipage}\hfill
\begin{minipage}[t]{0.26\textwidth}\centering
  \includegraphics[width=\textwidth]{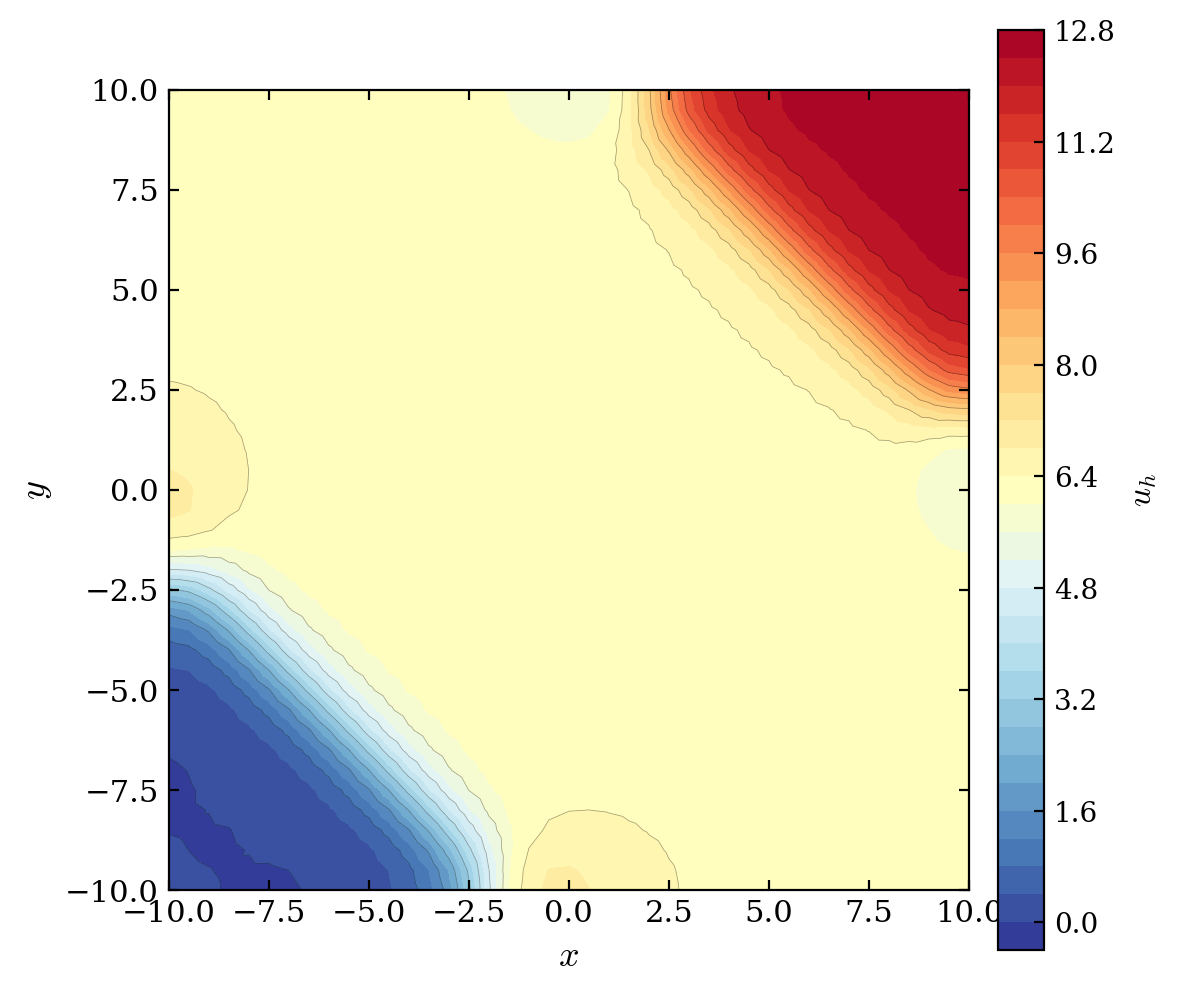}
  \caption*{(f) Contour, $t=10$}
\end{minipage}
\caption{Damped sine-Gordon ($\sigma=1$): progressive amplitude decay and front broadening due to energy dissipation.}
\label{fig:sg-damped}
\end{figure}

\begin{figure}[htbp]
\centering
\includegraphics[width=0.70\textwidth]{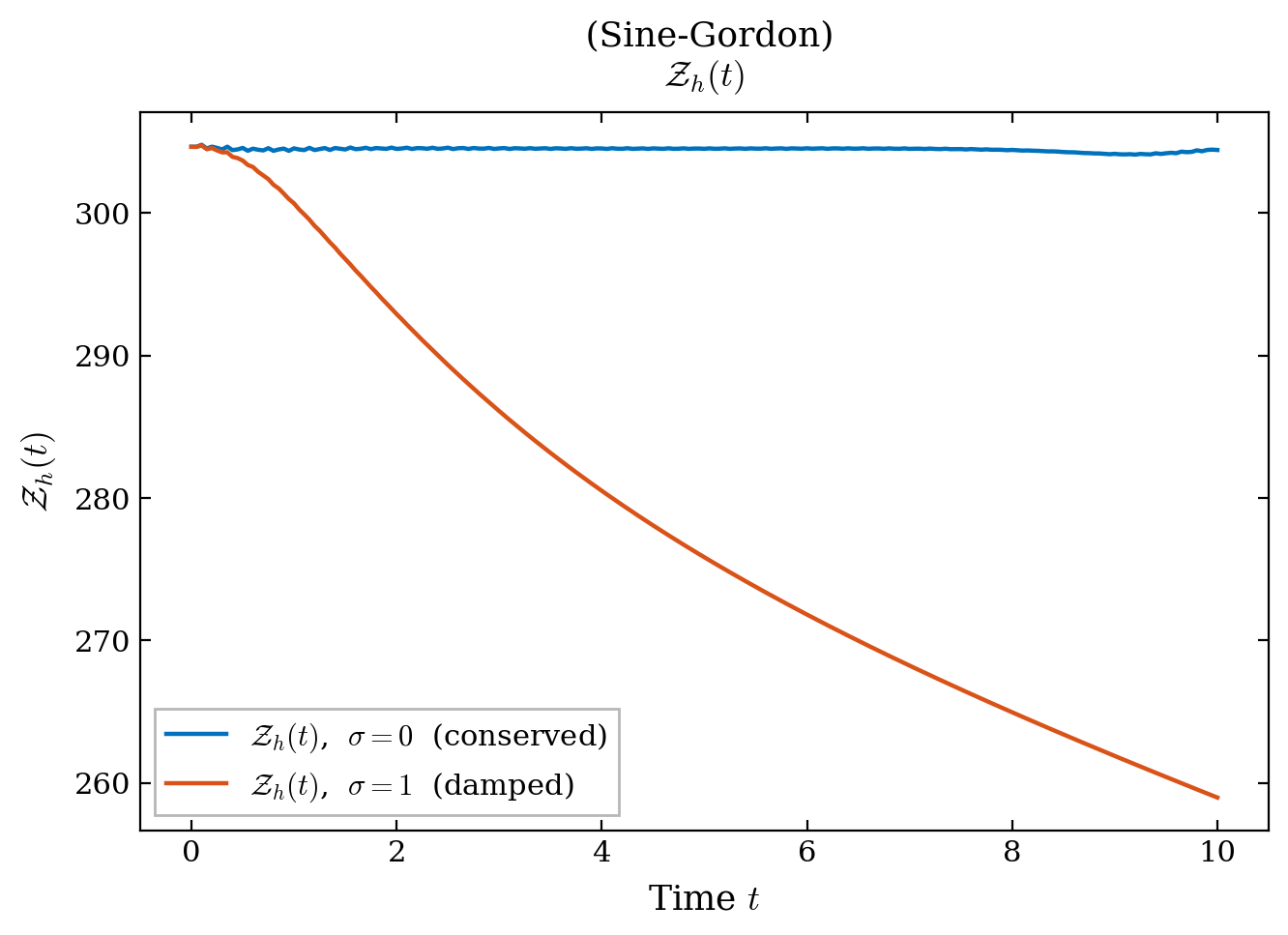}
\caption{Discrete Lyapunov functional $\mathcal{Z}_h(t)$:
conserved for $\sigma=0$, monotonically decaying for $\sigma=1$.}
\label{fig:sg-lyapunov}
\end{figure}
Although our analysis assumes homogeneous Dirichlet conditions,
the energy structure of Lemma~\ref{lem:energy} requires only
symmetry and non-negativity of ${\cal A}_h$, both of which hold
in the Neumann case.
 \section{Concluding Remarks}\label{sec:conclusion}

We have presented and analyzed a symmetric interior penalty discontinuous Galerkin (SIPG) method combined with the CN--BDF2 time-stepping scheme for the weakly damped semilinear wave equation. The chord-slope operator $G(a,b)$ preserves the exact discrete energy structure without requiring global Lipschitz continuity of the nonlinearity. We established existence and uniqueness of the fully discrete solution (Lemma~\ref{lem:existence}) and derived optimal a priori error estimates of order $\mathcal{O}(h^{k}+\tau^{2})$ in the discrete energy norm and $\mathcal{O}(h^{k+1}+\tau^{2})$ in $L^{2}$ (Theorems~\ref{thm:error}--\ref{thm:l2-error}). Numerical experiments confirm the theoretical rates for linear and nonlinear test problems and demonstrate that the discrete Lyapunov framework correctly captures long-time energy dissipation and conservation. Future work includes extending the analysis to strongly damped (Kelvin--Voigt) models, developing $hp$-adaptive strategies, and investigating uniform-in-time error estimates under the discrete Lyapunov framework.\\

\textbf{CRediT authorship contribution statement :}

Ajeet Singh: Writing – review and editing, Writing – original draft, Conceptualization, Methodology, Validation, Software. \\
Abhinav Jha: Writing—review and editing, supervision,  conceptualization, Methodology. Data availability.\\

\textbf{Data Availability Statement :}
No new data were created or analyzed in this study.\\

\textbf{Acknowledgment}
The work of AS has been supported by the IIT Gandhinagar Grant: IP/52012, and the work of AJ has been partially supported by the IIT Gandhinagar Internal Project: IP/52016 and INSPIRE Faculty Fellowship Research Grant: DST/INSPIRE/04/2024/000202.
\bibliographystyle{plain}
\bibliography{wave_IPDG}
\end{document}